\documentclass[11pt,a4paper,leqno]{article}
\usepackage{MeinArtikelStyle}

\title{Positive Harris recurrence for degenerate diffusions with internal variables and randomly perturbed time-periodic input}
\author{Simon Holbach\footnote{Fakult{\"a}t f{\"u}r Mathematik, Universit{\"a}t Bielefeld, Postfach 10 01 31, 33501 Bielefeld, Germany, e-mail: sholbach$@$math.uni-bielefeld.de. Some of the results featured in this article have been part of the author's PhD thesis \cite{ICHDiss} at Johannes Gutenberg-Universit{\"a}t Mainz.}}

\begin{document}
\maketitle
\begin{abstract}
We consider a multidimensional time-homogeneous dynamical system and add a randomly perturbed time-dependent deterministic signal to some of its components, giving rise to a high-dimensional system of stochastic differential equations, which is driven by possibly very low-dimensional noise. Equations of this type commonly occur in biology when modelling neurons or in statistical mechanics for certain Hamiltonian systems. We provide verifiable conditions on the original deterministic dynamical system under which the solution to the respective stochastic system features a point in the interior of its state space, which can be proved to be attainable by deterministic control arguments, and at which a local H{\"o}rmander condition holds. Together with a Lyapunov condition, it follows that the corresponding process is positive Harris recurrent. 
\end{abstract}

\small{\textbf{Keywords:} degenerate diffusion, periodic ergodicity, Harris recurrence, local H{\"o}rmander condition, deterministic control, dynamical system

}

\section{Introduction and main results}
\label{sect:intro}

Let $\tU\subset\R^N$ and $\tV\subset\R^L$. For some functions $f\colon \tU\times\tV \to \R^N$ and $g\colon \tU\times\tV \to \R^L$ and a signal $S\colon[0,\infty)\to\R^N$, consider the deterministic dynamical system
\begin{align}\label{eq:ODE}
	\begin{split}
		\dot x &= f(x,y) + S, \\
		\dot y &= g(x,y),
	\end{split}
\end{align}
where $\dot x$ and $\dot y$ are the time-derivatives of the time-dependent variables $x\colon[0,\infty)\to\tU$ and $y\colon[0,\infty)\to\tV$, respectively.

This system is divided into two groups of variables: The $N$ components of $x$, whose dynamics depend directly on the signal, and the $L$ components of $y$, which are affected by the signal only indirectly through the influence of $x$. Intuitively speaking, we can think of \eqref{eq:ODE} as a dynamical system with no intrinsic time-inhomogeneity, which then receives an additional time-dependent \emph{external input} in some of its variables, while the remaining variables merely describe an interior mechanism. This is why we sometimes refer to $x$ as the \emph{adjustable variable(s)} and $y$ as the \emph{internal variable(s)}. Note that the only source of time-inhomogeneity is indeed the signal -- if the system receives constant external input $S\equiv c \in \R^N$ (or none at all, i.e.\ $c=0$), it is homogeneous in time. Systems of this kind frequently arise in the context of neuroscience and statistical mechanics (see Examples \ref{ex:HH} and \ref{ex:rotors} below).

We construct a stochastic model by following the idea that the signal is not actually received in its original shape, but is subject to random perturbations by external noise (i.e.\ noise that is independent of the rest of the system). Using integral notation, we write \eqref{eq:ODE} as
\begin{align}\label{eq:ODEintegral}
	\begin{split}
		d x(t) &= f(x,y)(t)dt + S(t)dt, \\
		d y(t) &= g(x,y)(t)dt.
	\end{split}
\end{align}
and in order to perturb the signal, it seems natural to substitute the term $S(t)dt$ in \eqref{eq:ODEintegral} with the increment $dZ_t$ of a process taking values in some $\tW\subset\R^N$ and satisfying an SDE of the type
\[
dZ_t = [S(t)+b(Z_t)]dt + \sigma(Z_t)dW_t,
\]
where $W$ is an $M$-dimensional standard Brownian Motion, while $b\colon \tW \to \R^N$ and $\sigma\colon \tW \to \R^{N\times M}$ are suitable drift and diffusion coefficients. Note that this SDE can be viewed as a generalized Ornstein--Uhlenbeck type process with time-dependent mean-reversion level (think of $b(Z_t)=-\beta Z_t$ with $\beta\in(0,\infty)$). A particularly prominent special case is the classical signal in noise model (take $M=N=1$, $b\equiv 0$, and $\sigma\equiv 1$), which arises in a wide variety of fields including communication, radiolocation, seismic signal processing, or computer-aided diagnosis and has been the subject of extensive study.

Perturbing $S(t)$ randomly in this way leads to the stochastic dynamical system
\begin{align}\label{eq:SDE}
	\begin{split}
		dX_t &= f(X_t,Y_t)dt + dZ_t, \\
		dY_t &= g(X_t,Y_t)dt, \\
		dZ_t &= [S(t)+b(Z_t)]dt + \sigma(Z_t)dW_t,
	\end{split}
\end{align}
with state space
\[
\tE:=\tU\times\tV\times\tW\subset\R^{N+L+N}.
\]
This system can be thought of as degenerate in the following sense: Firstly, the equation for $Y$ does not incorporate the driving Brownian Motion $W$ explicitly, making it rather unclear which effect noise has on these components. Secondly, the dimension $M$ of the driving Brownian Motion can (and will usually) be much lower than the dimension $N+L+N$ of the system. This is why we call a stochastic process satisfying a system of SDEs of the type \eqref{eq:SDE} a \emph{degenerate diffusion with internal variables and randomly perturbed time-inhomogeneous deterministic input}.

We now have three groups of variables: The entirely autonomous external input governed by $dZ_t$ (the "noisy signal"), the components of $X$ that are directly adjusted by the noisy signal, and the components of the internal variable $Y$, whose dynamics are only indirectly affected by noise, since the respective differential equations incorporate neither $Z$ nor the driving Brownian Motion $W$ explicitly. Note that for this reason $Y$ is conditionally deterministic given $X$ and has continuously differentiable trajectories.

The system \eqref{eq:SDE} is a generalization of the one introduced in equation (18) of \cite{HLT2}, which is a probabilistic version of a class of dynamical systems that are well-known in the mathematical modelling of neurons (see Example \ref{ex:HH} below). In \cite{ICHLAN} (which can be viewed as a statistical companion article to the present one), we study the problem of estimating parameters that are present in the signal $S$ (in particular its periodicity). Before we explain the focus of the current article, let us introduce some major examples.

\begin{example}\label{ex:HH}
	Let $N=1$, $L=3$, $\tU=\R$, $\tV=[0,1]^3$ and consider the coefficient functions
	\[
	f(x,y)=-36y_1^4(x+12)-120y_2^3y_3(x-120)-0.3(x-10.6)
	\]
	and
	\[
	g(x,y)=\begin{pmatrix} \alpha_1(x)(1-y_1)-\beta_1(x)y_1 \\ \alpha_2(x)(1-y_2)-\beta_2(x)y_2 \\ \alpha_3(x)(1-y_3)-\beta_3(x)y_3	\end{pmatrix}
	\]
	with
	\[
	\begin{array}{llllll}
	\alpha_1(x) &=& \begin{cases} \frac{0.1-0.01x}{\exp(1-0.1x)-1}, & x\neq 10, \\ 0.1, & \text{else},\end{cases}& \beta_1(x)& = & 0.125\exp(-x/80),  \vspace{3mm} \\
	\alpha_2(x) &=& \begin{cases} \frac{2.5-0.1x}{\exp(2.5-0.1x)-1}, & x\neq 25, \\ 1, & \text{else},\end{cases}& \beta_2(x)& = & 4\exp(-x/18), \vspace{2.5mm} \\
	\alpha_3(x) &=& 0.07\exp(-x/20),& \beta_3(x)& = & \frac{1}{\exp(3-0.1x)+1}. \\
	\end{array}
	\]
	for all $(x,y)=(x,y_1,y_2,y_3)^\top\in\R\times [0,1]^3$.
	The corresponding dynamical system \eqref{eq:ODE} is known as the \emph{Hodgkin--Huxley system} and it was first introduced by Hodgkin and Huxley in 1952 (see \cite{HoHu}, note however that we use the slightly different model constants from \cite{Izhi}) with the aim of describing the initiation and propagation of action potentials in the cell membrane of a neuron in response to an external stimulus. While $x$ is the membrane potential itself (usually labelled $V$ in the literature), the internal variables $y_1$, $y_2$, and $y_3$ (commonly denoted by $n$, $m$, and $h$) correspond to the ionic mechanism underlying its evolution. The two predominant ion currents in the cell membrane are import of sodium $Na^+$ and export of potassium $K^+$ through the membrane. Each of the internal variables signifies the probability that a specific type of gate in the respective ion channel is open at a given time. It is for this reason that $n$, $m$, and $h$ are often called gating variables. In the context of this model, the signal $S$ represents the dendritic input that the neuron receives from a large number of other neurons, transported by an even larger number of synapses located on the respective dendritic tree. The resulting "total dendritic input" can then be thought of as an average of interdependent and repeating similar currents, which is why $S$ is usually assumed to be periodic (or even constant). When modelling neurons, particular interest lies in the typical spiking behaviour of the membrane potential, a feature that is commonly agreed upon to be adequately described by the Hodgkin--Huxley model. For a more detailed modern introduction, interpretation, and an in-depth comparison with other neuron models, see for example \cite{Izhi} and \cite{Dest}.
	
	Adding noise in the sense of \eqref{eq:SDE} by choosing	$\sigma\in C^\infty(\tW)$ and $b(Z_t)=- \beta Z_t$ with $\beta\in(0,\infty)$, we acquire the so-called \emph{stochastic Hodgkin--Huxley model (with mean reverting Ornstein--Uhlenbeck type input)}. It was first introduced and studied by H{\"o}pfner, L{\"o}cherbach, and Thieullen in the series of the three papers \cite{HLT1}, \cite{HLT2}, and \cite{HLT3}. The constant $\beta$ is determined by the so-called time constant of the membrane, which represents spontaneous voltage decay not related to the input. For many types of neurons, the time constant is known from experiments (see \cite{Ditlevsen}). A degree of freedom lies in the choice of the diffusion coefficient $\sigma$, which reflects the nature of the influence of noise. In the past, mean reverting Ornstein--Uhlenbeck type equations with various volatilities have been used to model the membrane potential itself (see for example \cite{Lansky} or \cite{HoBio}), and in a sense our stochastic Hodgkin--Huxley model can be viewed as a refinement of this kind of model. If $\sigma$ is Lipschitz continuous, existence of a unique non-exploding strong solution follows from the same arguments as in \cite[Proposition 1]{HLT1} and \cite[Proposition 2]{HLT2}.
	
	In \cite{HLT3}, the authors prove that if the external equation is of classical Ornstein--Uhlenbeck type
	\[
	dZ_t = [S(t)-\gamma Z_t]dt + \sigma dW_t
	\]
	with constant $\sigma\in(0,\infty)$ and $\tW=\R$, or (when the signal $S$ is non-negative) of Cox-Ingersoll-Ross type\footnote{Note that a Cox-Ingersoll-Ross type equation for $Z$ is not contained in our model, since we require its state space to be the full euclidean space and the diffusion coefficient to be defined everywhere on it. This is only of major importance for Section \ref{sect:control}, though.}
	\[
	dZ_t = [a+S(t)-\gamma Z_t]dt + \sqrt{Z_t} dW_t
	\]
	with $2a\in(1,\infty)$ and $\tW=[0,\infty)$, the solution to the stochastic Hodgkin--Huxley system is positive Harris recurrent (see \cite[Theorem 2.7]{HLT3}). Moreover, they quantify the typical spiking behaviour: Almost surely, there are infinitely many spikes but also infinitely many periods of the signal in which no spike occurs (see \cite[Theorem 2.8]{HLT3}). Harris recurrence then enables them to prove a Glivenko-Cantelli type theorem for the interspike intervals (see \cite[Theorem 2.9]{HLT3}).
	
	Analogously, one can introduce stochastic versions of simpler neuron models such as the FitzHugh-Nagumo model (see \cite[equations (4.11) and (4.12)]{Izhi}) or the Morris-Lecar model (see \cite{Morris} or, for a modern version, \cite{Rinzel}).
\end{example}

\begin{example}\label{ex:rotors}
	Systems of coupled oscillators are particularly intuitive Hamiltonian systems and several different stochastic models have recently been subject to research (see e.g\ \cite{Hairer}, \cite{Cuneo}, \cite{Cuneo2}, \cite{Rey-Bellet}, \cite{Rotoren2018}). We illustrate the connection to our model with a simple example.
	
	Let us think of three rotors, each given by their angle $q_i(t)\in\R$ and momentum $p_i(t)\in\R$ at the time $t\in[0,\infty)$ for each $i\in\{1,2,3\}$. Assuming their respective masses to be all equal to $1$ and not taking into account units, the laws of classical mechanics imply
	\begin{equation}\label{eq:rotors0}
	\dot q_i = p_i \quad \text{for all $i\in\{1,2,3\}$.}
	\end{equation}
	We suppose that these rotors are coupled in row, i.e.\
	\begin{align}\label{eq:rotors}
		\begin{split}
			\dot p_1 &= w_1(q_2-q_1)-u_1(q_1), \\
			\dot p_2 &= -[w_1(q_2-q_1) + w_3(q_2-q_3)]-u_2(q_2), \\
			\dot p_3 &= w_3(q_2-q_3)-u_3(q_3),
		\end{split}
	\end{align}
	where $w_1,w_2,w_3\colon\R\to\R$ and $u_1,u_2,u_3\colon\R\to\R$ are related to interaction potentials and pinning potentials, respectively. A classical model is the one that arises if we let one or both of the outer rotors receive external torques and interact with Langevin type heat baths. In order to give a mathematical description of this, we fix	$i\in\{1,3\}$ for the remainder of this paragraph. Applying an external time-dependent torque $S_i\colon[0,\infty)\to\R$ to the $i$-th rotor means expanding the equation for $p_i$ to
	\[
	dp_i = \left[w_i(q_2-q_i)-u_i(q_i)\right]dt +S_i dt,
	\]
	which turns \eqref{eq:rotors0} and \eqref{eq:rotors} into a system like \eqref{eq:ODE}. On top of that, we want to add interaction with a heat bath, i.e.\ for a temperature $\tau_i\in(0,\infty)$ and a dissipation constant $\delta_i\in(0,\infty)$, the equation for $p_i$ is further expanded to
	\begin{align*}
		dp_i &= \left[w_i(q_2-q_i)-u_i(q_i)\right]dt + S_idt - \delta_i p_i dt + \sqrt{2\delta_i\tau_i}dW^{(i)}_t \\
		&=\left[w_i(q_2-q_i)-u_i(q_i)-\delta_i p_i\right]dt  + \left[S_i dt+ \sqrt{2\delta_i\tau_i}dW^{(i)}_t\right],
	\end{align*}
	where the last term in parentheses is the total sum of external influences. Following the spirit of \eqref{eq:SDE}, we may replace this term with the increments of a more general random perturbation of the torque: We take
	\[
	dp_i =\left[w_i(q_2-q_i)-u_i(q_i)-\delta_i p_i\right]dt  + dZ^{(i)}_t
	\]
	with
	\[
	dZ^{(i)}_t=\left[S_i(t)+b_i(Z^{(i)}_t)\right]dt + \sigma_i(Z^{(i)}_t)dW^{(i)}_t
	\]
	for some diffusion coefficient $\sigma_i\colon\R\to\R$ and a drift $b_i\colon\R\to\R$. What we end up with is indeed a degenerate diffusion with internal variables and randomly perturbed time-inhomogeneous deterministic input as in \eqref{eq:SDE}. If only the first rotor in the chain receives an external input, the dimensions are $M=N=1$ and $L=5$, $\tE=\R^7$. If both of the outer rotors receive an external input, the dimensions are $M=N=2$ and $L=4$, $\tE=\R^8$.
	
	In the present text, this example will mainly serve as a hint at the potential of studying the general system \eqref{eq:SDE} beyond the realm of neuroscience -- which the theory displayed in this article is predominantly inspired by. Only certain parts of our results are applicable to this stochastic rotor system, but there are enough connections to spawn optimism for future work that may unify some of the approaches used in different settings.
	
	Let us stress that the differences between the eight-dimensional system above with input on both ends of the chain and the six-dimensional system that is discussed in \cite{Cuneo} mainly lie in the role of the external equations and in time-inhomogeneity: Forcing the latter into our notational framework, the equations for $X$ and $Y$ would remain unaltered, while the external equations would read
	\[
	dZ^{(i)}_t=S_i dt+ \sqrt{2\delta_i\tau_i}dW^{(i)}_t, \quad i \in\{1,2\},
	\]
	where the coefficient functions do not depend on $Z^{(i)}_t$. Therefore, including these variables separately into the system is rendered obsolete. In addition to that, with the external torque in \cite{Cuneo} being constant, their entire system is homogeneous in time -- in contrast to ours. Another important difference is that while the system in \cite{Cuneo} lets the angle take values in the torus $\R \mod 2\pi$, the state space for \eqref{eq:SDE} has to be a proper subset of $\R^{N+L+N}$ (without any topological identifications). Therefore, the system we describe here describes oscillators but not rotors in a strict sense.
	
	For the six-dimensional system from \cite{Cuneo}, the authors' main result (Theorem 1.3) states the existence of smooth transition densities, unique existence of an invariant measure with a smooth density, and finally establishes a subgeometric convergence rate of the semi-group to the invariant measure.
\end{example}

\begin{example}\label{ex:toy}
	For a convenient toy example, let $\tU=\tW=\R^N$ with $N\in\{1,2\}$ and $\tV=\R^L$. Let
	\[
	f(x,y)=-\big(1+\abs{y}^2\big) \big(\abs{x}^2-1\big) x=-\big(1+\abs{y}^2\big) \frac14 \nabla_x \big(\abs{x}^4-\abs{x}^2\big) \quad \text{for all $x\in\R^N$, $y\in\R^L$.}
	\]
	This means that the $x$-variables are subject to a double well potential (for $N=1$) or a mexican hat potential (for $N=2$) where the slope is determined by the current $y$-position, but is bounded away from zero.
	
	For $L=1$, let 
	\begin{equation}\label{eq:toyg1}
	g(x,y)=-\big(1+\abs{x}^2\big) y + \sin\big(\abs{x}^2-1\big) \quad \text{for all $x\in\R^N$, $y\in\R$.}
	\end{equation}
	By the first summand, the $y$-variable experiences exponential decay with a rate that depends on $x$, but is again bounded away from zero. The second summand represents some bounded $x$-dependent source that vanishes at the bottom of the wells or of the hat.
	
	A different choice for $g$ in the case $N=1$ (but for any $L$) is given by
	\begin{equation}\label{eq:toyg2}
	g(x,y)= 
	(h(x) -y_1,  y_1-y_2, \ldots , y_{L-1}-y_L)^\top
	\quad \text{for all $x\in\R$, $y\in\R^L$,}
	\end{equation}
	where $h\colon\R\to\R$ is a bounded smooth function with $h(x)=x^2$ for all $x\in (-2,2)$. Here, the dependence of each component on the variables takes a cascade structure.

	For the external equation, we take $\tW=\R^N$, and
	\begin{equation}\label{eq:toysigma}
	dZ_t=[S(t)-\beta Z_t]dt+\sigma(Z_t) dW_t
	\end{equation}
	where $\beta\in\R^{N\times N}$ is strictly positive definite and all entries of $\sigma\colon\R^N\to\R^{N\times M}$ are bounded.
\end{example}

The objective of this article is to find verifiable conditions on \eqref{eq:ODE} and the way in which the external equation of \eqref{eq:SDE} perturbs the signal, under which positive Harris recurrence of $(X,Y,Z)$ can be established in spite of its apparent degeneracy. Positive Harris recurrence is the foundation of a wide class of limit theorems for Markov processes, most prominently the Ratio Limit Theorem and Orey's Ergodic Theorem. Classical references for the theory of recurrence in the sense of Harris include \cite{Harris}, \cite{Revuz}, and \cite{Nummelin} in the discrete-time setting and \cite{Azema} in the continuous-time case. Of course, in order to make use of the respective theory, one has to look for some time-homogeneous substructure, which suggests the basic assumption that the signal function $S$ should be periodic. In this setting, one can modify known tools and provide a strategy to prove positive Harris recurrence (see Theorem \ref{thm:HLT} below, c.f.\ \cite{Mattingly}, \cite{HLT3}). Our main results provide conditions on the deterministic system \eqref{eq:ODE} under which this strategy can be applied to the stochastic system \eqref{eq:SDE}. As one of the respective ingredients, Section \ref{sect:control} shows:
\begin{framed}
	\noindent\textbf{Result 1.} If $T$ and $T^*$ are incommensurable and $(x^*,y^*)$ lies on the trajectory of an attractive $T^*$-periodic solution to the deterministic system \eqref{eq:ODE} with appropriate periodic input signal $S$, the solution $(X,Y,Z)$ to the stochastic system \eqref{eq:SDE} with \emph{any} $T$-periodic signal $S$ will almost surely visit every neighborhood of $(x^*,y^*,z^*)$ with $z^*\in\tW$, even if time is restricted to multiples of $T$. The only assumption on the external equation is pointwise surjectivity of $\sigma(\cdot)\in\R^{N\times M}$, no assumption is needed that noise is small in any sense.
\end{framed}
\noindent This is obtained via control arguments relying on the Support Theorem for diffusions. The main corresponding result in precise terms are Theorems \ref{thm:controlsimple}, \ref{thm:control}, and \ref{thm:controlperiodic}. The second main ingredient is the content of Section \ref{sect:lwh} and can be summarized as follows:
\begin{framed}
	\noindent\textbf{Result 2.} If $\sigma$ is nice, simple and verifiable conditions on the derivatives of $f$ and $g$ imply a local H{\"o}rmander condition, which in turn yields local transition densities for $(X,Y,Z)$.
\end{framed}
\noindent This is the content of Section \ref{sect:lwh}. We give essentially two different sets of conditions on the derivatives of $f$ and $g$ (Theorems \ref{cor:stardiag} and \ref{thm:cascade}), both of which have a clear interpretation and can be used for applications. The third main ingredient for the strategy arising from Theorem \ref{thm:HLT} is a Lyapunov condition that is more reliant on ad hoc arguments and that is briefly touched upon in Remark \ref{rem:lyapunov}.

\section{Precise setting and general technique}

For the sake of better readability, we will often use the short notation
\[
	\Phi=(\phi,z)=(x,y,z)\in\tU\times\tV\times\tW=\tE.
\]
With $\phi(t)=(x(t),y(t))$ for all $t\in[0,\infty)$, we can then rewrite \eqref{eq:ODE} as
\begin{equation}\label{eq:ODEphi}
\dot\phi(t)=F_S(t,\phi(t)),
\end{equation}
where
\begin{equation}
F_S\colon [0,\infty)\times\tU\times\tV\to\R^{N+L}, \quad (t,\phi) \mapsto \begin{pmatrix} f(\phi) + S(t) \\ g(\phi) \end{pmatrix}.
\end{equation}
If it exists, we write $\phi[\phi_0,S](t)$ for the unique solution of \eqref{eq:ODEphi} at time $t\in[0,\infty)$ with starting condition $\phi(0)=\phi_0\in \tU\times\tV$. We also write 
\begin{equation}\label{eq:FohneS}
	F(\phi):=F_{S\equiv0}(t,\phi)=\begin{pmatrix} f(\phi) \\ g(\phi) \end{pmatrix} \quad \text{for all $(t,\phi)\in[0,\infty)\times\tU\times\tV$.}
\end{equation}
While we will often work with different signals that are fed into the deterministic system \eqref{eq:ODEphi}, for the random system we fix one reference signal
\[
	S_0\colon[0,\infty)\to\R^N.
\] 
Similarly as above, if we set $\Phi_t:=(X_t,Y_t,Z_t)$ for all $t\in[0,\infty)$, we can express the diffusion equation \eqref{eq:SDE} as
\begin{equation}\label{eq:SDEPhi}
d\Phi_t=B(t,\Phi_t)dt+\Sigma(\Phi_t)dW_t
\end{equation}
with drift
\begin{equation}\label{eq:B}
B\colon [0,\infty)\times \tE \to \R^{N+L+N}, \quad (t,\Phi)\mapsto \begin{pmatrix}
f(\phi) +S_0(t)+b(z)\\
g(\phi) \\
S_0(t)+b(z)
\end{pmatrix},
\end{equation}
and diffusion matrix
\begin{equation}\label{eq:Sigma}
\Sigma \colon \tE \to \R^{(N+L+N)\times M}, \quad \Phi\mapsto \begin{pmatrix}
\sigma(z)\\
0_{L\times M} \\
\sigma(z)
\end{pmatrix}.
\end{equation}
We fix a probability space $(\Omega,\cA,\P)$ and on it an $M$-dimensional Brownian Motion $W$.

\begin{remark}
The most practically relevant results are acquired in the case $M=N$, but distinguishing notationally between the space dimension and the dimension of the driving Brownian motion can make some arguments a little clearer and easier to read.
\end{remark}

The following basic assertions will be assumed to hold throughout this article.
\begin{enumerate}
	\item[\textbf{(A0)}]
	For all deterministic starting points $\Phi_0 \in \tE$, the equation \eqref{eq:SDEPhi} has a unique non-explosive strong solution $(\Phi_t)_{t\in[0,\infty)}$ on $(\Omega,\cA)$ under $\P$.
	\item[\textbf{(A1)}] $S_0$ is periodic with periodicity $T\in(0,\infty)$.
	\item[\textbf{(A2)}] All of the coefficient functions $F$, $G$, $S_0$, $b$, and $\sigma$ have derivatives of any order with respect to any of their variables.
	\item[\textbf{(A3)}] There is a strictly increasing sequence $(G_n)_{n\in\N}$ of bounded open convex subsets of $\tE$ with $\bigcup_{n\in\N}\overline{G_n}=\tE$ such that for any $\Phi_0 \in \tE$, the following properties hold $\P$-almost surely:
	\begin{enumerate}
		\item If $\Phi_0\in \d\tE$, then $\inf \{t\in(0,\infty) \,|\, \Phi_t \in \interior(\tE) \}=0$.
		\item If $\Phi_0\in \overline{G_n}\setminus G_{n+1}$ for some $n\in\N$, then $\inf\{t\in(0,\infty)\,|\, \Phi_t \in G_{n+1} \}=0$.
		\item We have $\inf\{ t\in(0,\infty) \,|\, \Phi_t \notin \overline{G_n}\} \xrightarrow{n\to\infty}\infty$.
	\end{enumerate}
\end{enumerate}

Note that in view of (A3), $\tU$, $\tV$ and $\tW$ are necessarily convex, $\sigma$-compact sets with non-empty interior. The condition (A3) can be checked for the stochastic Hodgkin--Huxley model (see \cite[page 533]{HLT3}), and in the oscillator system from Example \ref{ex:rotors} it becomes trivial since $\tE=\R^{N+L+N}$, which is also the case in Example \ref{ex:toy}.

Under (A1), the entire drift of \eqref{eq:SDE} is $T$-periodic in time, and consequently the corresponding transition semi-group $\left(P_{s,t}\right)_{t> s\ge 0}$ has the property
\begin{equation}\label{eq:semigroupperiodic}
P_{s+kT,t+kT}=P_{s,t} \quad \text{for all $t>s\ge 0$ and $k\in\N$.}
\end{equation}
This implies that the $T$-grid chain $(\Phi_{nT})_{n\in\N_0}$ is a discrete time Markov chain with no time-inhomogeneity. As such, its recurrence properties can be studied with classical methods. Our general method is summarized in the following theorem, which is in its essence a time-inhomogeneous variant of \cite[Theorem 2.5]{Mattingly} (compare \cite[Theorem 15.0.1]{MeynTweedie}) which is also the foundation of \cite[Theorem 2.2]{HLT3}. 

\begin{thm}\label{thm:HLT}
	Grant the following assumptions:
	\begin{enumerate}
		\item[\emph{\textbf{(I)}}] There is a \emph{Lyapunov function for the $T$-grid chain}, i.e. there are a measurable function $V\colon\tE\to[1,\infty)$ and a compact set $K\subset\tE$ such that
		\[
			V(\Phi)\xrightarrow{\abs{\Phi}\to\infty}\infty, \quad \sup_K P_{0,T}V <\infty, \quad \text{and} \quad \inf_{\tE \setminus K}(1-P_{0,T})V>0.
		\]
		\item[\emph{\textbf{(II)}}] There is a point $\Phi^*\in\interior(\tE)$ that is \emph{$T$-attainable}, i.e. for any starting point $\Phi_0\in \tE$ and any $\eps>0$ there is an $n=n(\Phi_0,\eps)\in\N$ such that
		\[
			P_{0,nT}\big(\Phi_0, B_\eps(\Phi^*)\big) >0.
		\]	
		\item[\emph{\textbf{(III)}}] There is an open neighborhood $\cU^*$ of $\Phi^*$ such that for all $n\in\N$ the kernel $P_{0,nT}$ \emph{admits a lower semi-continuous local Lebesgue-density} $p_{0,nT}\colon \tE \times \cU^*\to[0,\infty)$, i.e.
		\[
			P_{0,nT}(\Phi_0,B)=\int_B p_{0,nT}(\Phi_0,\Psi)d\Psi \quad \text{for all $\Phi_0\in\tE$ and measurable sets $B\subset\cU^*$.}
		\]
	\end{enumerate}
	Then the $T$-grid chain $(\Phi_{nT})_{n\in\N_0}$ is positive Harris recurrent.
\end{thm}

Condition (I) is classical and can often be treated with ad hoc arguments for the specific system. We only deal with it explicitly in Remark \ref{rem:lyapunov} below. For (II) and (III), we follow a strategy that is also used in \cite{HLT3}: In Section \ref{sect:control}, we present stability properties of \eqref{eq:ODEphi} under which we can prove (II) by constructing suitable control paths and applying the Support Theorem. In Section \ref{sect:lwh}, we investigate properties of the coefficient functions that allow to check (III) via a local variant of H{\"o}rmander's condition.

\begin{remark}\label{rem:HLT}
1.) It is well-known that if the $T$-grid chain is positive Harris recurrent, then the same is true for the $C([0,T];\tE)$-valued path-segment chain $\big((\Phi_t)_{t\in[(n-1)T,nT]}\big)_{n\in\N}$ and the $[0,T]\times\tE$-valued time-space process $(t\mod T,\Phi_t)_{t\in[0,\infty)}$.

2.) Note that by Lemma \ref{lem:attainthesame} below, (II) is slightly weaker than "attainability in a sense of deterministic control" as in \cite[Definition 2.1]{HLT3}.

3.) In \cite{HLT3}, Theorems 1 and 2 from \cite{HLT2} are used to check (III). In their original statements, these theorems treat regularity of $p_{0,nT}$ only \emph{separately} in $\tE$ and $\cU^*$, but it is crucial for the proof of Theorem \ref{thm:HLT} that lower semi-continuity holds \emph{jointly} in $\tE\times\cU^*$. Since $\cU^*\ni\Psi\mapsto p_{0,nT}(\Phi,\Psi)$ is equi-continuous with respect to $\Phi\in\tE$ (as noted at the end of the proof of \cite[Theorem 1]{HLT2}), this can be fixed in that case.

4.) Note that the Lyapunov condition (I) is slightly weaker than its counterpart in \cite[Theorem 2.5]{Mattingly} where
\begin{equation}\label{eq:Lyastrong}
P_{0,T}V\le \alpha V + \beta \quad \text{for some $\alpha\in(0,1)$ and $\beta\in(0,\infty)$,}
\end{equation}
which does not only yield positive Harris recurrence but even exponential ergodicity.
\end{remark}

\begin{remark}\label{rem:lyapunov}
In \cite{ICHDiss}, we give variants of classical drift conditions under which a Lyapunov function with the property \eqref{eq:Lyastrong} can be constructed. Possible choices for the external equation include classical multi-dimensional Ornstein-Uhlenbeck processes (see \cite[Example 2.15]{ICHDiss} for more details). The relevant conditions on $f$ and $g$ cover the stochastic Hodgkin--Huxley model in particular, meaning that the conclusion of \cite[Theorem 2.7]{HLT3} is immediately improved from just positive Harris recurrence to exponential ergodicity (which is not commented on in \cite{HLT3}, even though the Lyapunov function given there in the proofs of Propositions 2.5 and 2.6 is not essentially different from the one in \cite[Example 2.16]{ICHDiss}). For the toy model from Example \ref{ex:toy}, \cite[Theorem 2.11]{ICHDiss} yields the existence of a Lyapunov function with the property \eqref{eq:Lyastrong}. However, these methods fail for the oscillator model from Example \ref{ex:rotors}. For the rotor system from \cite{Cuneo} with no external equations, the authors construct a Lyapunov function yielding stretched exponential ergodicity (see \cite[Proposition 2.2 and Section 4]{Cuneo}). The key to their result lies in understanding the slow dynamics of the momentum $p_2$ of the middle rotor and using appropriate averaging techniques.
\end{remark}

In order to make the strategy more transparent, we would like to give a short proof of Theorem \ref{thm:HLT} condensing the reasoning used in \cite{HLT3} to adjust \cite[Theorem 2.5]{Mattingly} to the time-periodic setting. We also use this occasion to fix some minor gaps and inaccuracies in \cite{HLT3}. Let us stress, though, that we make no claims of originality.

\begin{proof}[Proof of Theorem \ref{thm:HLT}]
The main idea is to combine (II) and (III) in order to prove that the set $K$ from (I) is small in the sense of Nummelin and then apply \cite[Theorem 3.7(v) and Proposition 5.10]{Nummelin}. As the dependence on the initial condition in (II) is problematic for this approach, we get rid of it by switching to the resolvent kernel
\begin{equation*}\label{eq:samplekernel}
R:=\sum_{n=1}^\infty 2^{-n}P_{0,nT}= \sum_{n=1}^\infty 2^{-n}P_{0,T}^n
\end{equation*}
corresponding to the resolvent chain, i.e.\ the Markov chain arising from sampling the grid chain at independent geometric times. The grid chain is positive Harris recurrent if and only if the resolvent chain is, and we shift our focus to the latter for the rest of this proof.

By (II), we have
\[
	R\big(\Phi_0,B_\eps(\Phi^*)\big) >0 \quad \text{for all $\Phi_0\in \tE$ and $\eps>0$},
\]
and it is easy to see that (I) remains valid with $R$ in place of $P_{0,T}$, while (III) yields (via Fatou's Lemma) that $R$ admits a lower semi-continuous local density $r\colon \tE \times \cU^*\to[0,\infty)$.\footnote{Note that \cite[Corollary 4.2]{HLT3} claims without proof that smoothness in the second variable is also passed on from $p_{0,nT}$ to $r$. It is not clear why this should be true in general, but luckily this is irrelevant for the proof of Theorem \ref{thm:HLT}.} 
The above implies
\[
	\int_{\cU^*}r(\Phi^*,\Psi)d\Psi=R\big(\Phi^*,\cU^*\big)>0
\]
and thus the existence of $\Psi^*\in\cU^*$ with $r(\Phi^*,\Psi^*)>0$. By lower semi-continuity there are open neighborhoods $C\subset\tE$ of $\Phi^*$ and $D\subset\tE$ of $\Psi^*$ such that $r$ is bounded away from zero on $C\times D$. This can be translated into the minorization condition
\[
	R\ge \alpha 1_C \otimes \cU_D,
\]
where $\alpha>0$ and $\cU_D$ denotes the uniform distribution on $D$. By attainability of $\Phi^*$ and by lower-semi-continuity (and again Fatou's Lemma) we have
\[
	\beta:=\inf_{\Phi_0\in K}R(\Phi_0,C)>0.
\]
Since $K$ is visited infinitely often by the resolvent chain as a consequence of (I), Borel-Cantelli now yields that so is $C$. As a byproduct of this, there automatically exists an irreducibility (even invariant) measure for $R$ that charges $C$ (compare \cite[Proof of Theorem 3]{Nummelinpaper}). Taken all together, $C$ is a small set that is visited infinitely often, and hence \cite[Theorem 3.7(v)]{Nummelin} yields that the resolvent chain is Harris recurrent. We denote its unique (up to a multiplicative constant) invariant measure by $\mu$. As $K$ is visited infinitely often, it follows from \cite[Proposition 2.4]{Nummelin} that $\mu(K)>0$. Since
\[
R^2\ge R\circ(\alpha 1_C \otimes \cU_D)\ge \beta \alpha 1_K \otimes \cU_D,
\]
this means that $K$ is also a small set with respect to $R$ and hence (I) and \cite[Proposition 5.10]{Nummelin} yield that the resolvent chain is positive Harris recurrent.
\end{proof}




\section{Deterministic control}\label{sect:control}

In this section, we will present a technique that uses stability properties of the deterministic dynamical system \eqref{eq:ODE} in order to prove condition (II) of Theorem \ref{thm:HLT} via the Support Theorem for diffusions (see Lemma \ref{lem:attainthesame} below).

We work under the following assumption.
\begin{enumerate}
	\item[\textbf{(C1)}] \textbf{Non-degeneracy of $Z$:} At each point $z\in\tW$, the linear mapping that is represented by $\sigma(z)\in\R^{N\times M}$ is surjective. Thus, $\sigma(z)$ has a right inverse, and we simply denote it by $\sigma^{-1}(z)\in\R^{M\times N}$.
\end{enumerate}
Note that surjectivity of the linear mapping $\sigma(z)\colon \R^M\to\R^N$ implies $M\ge N$, so (C1) guarantees that the external equation for $Z$ is not degenerate itself (compare condition (H2) in Section \ref{sect:lwh}). Also note that $\sigma^{-1}(z)$ depends continuously on $z\in\tW$, as $\sigma(\cdot)$ is continuous and so is the function that maps a surjective matrix to its right inverse.

Since we want to apply the Support Theorem, let
\begin{equation}\label{eq:stratodrift}
\tilde b\colon \tW\to\R^N,\quad z\mapsto b(z) - \frac12 \sum_{i=1}^N \sum_{j=1}^M \sigma_{i,j}(z)  \begin{pmatrix} \d_{z_i}\sigma_{1,j}(z)\\ \vdots \\ \d_{z_i}\sigma_{N,j}(z)\end{pmatrix},
\end{equation}
denote the Stratonovich version of the drift $b$ with respect to $\sigma$, i.e.\ $\tilde b$ is the drift coefficient such that $Z$ satisfies
\[
dZ_t=[S_0(t)+\tilde b (Z_t)]dt+\sigma(Z_t)\circ dW_t,
\]
where "$\circ \, dW_t$" indicates that we interpret the stochastic integral in the sense of Stratonovich instead of It\={o}. Hence, if we define $\tilde B$ in the same way as $B$ in \eqref{eq:B}, only replacing $b$ with $\tilde b$, we obviously get the Stratonovich drift for the entire system \eqref{eq:SDE}, i.e.\
\[
d\Phi_t=\tilde B (t,\Phi_t)dt+\Sigma(\Phi_t)\circ dW_t.
\]

\begin{lem}\label{lem:attainthesame}
	A point $\Phi^* \in \interior(\tE)$ is $T$-attainable in the sense of condition (II) of Theorem \ref{thm:HLT} if and only if for any starting point $\Phi_0\in \tE$ there are $\dot h \in \L^2_{\mathrm{loc}}\big([0,\infty);\R^M\big)$ and a solution $\Psi\in C([0,\infty);\tE)$ to the deterministic integral equation
	\begin{equation}\label{eq:controlsystem}
		d\Psi(t)=\tilde B \big(t,\Psi(t)\big)dt+\Sigma\big(\Psi(t)\big) \dot h(t) dt \quad \text{for all $t\in[0,\infty)$,}
	\end{equation}
	such that
	\begin{equation}\label{eq:controlconditions}
		\Phi_0=\Psi(0) \quad \text{and} \quad \Phi^*\in \overline{\Psi(\N T)}.
	\end{equation}
	A solution to \eqref{eq:controlsystem} is called a \emph{control path}.
\end{lem}

\begin{proof}
This is an immediate consequence of the Support Theorem and the fact that $S_0$ is $T$-periodic. A variant of the Support Theorem that applies to diffusions on state spaces with the property (A3) is given in \cite[Theorem 3.1]{HLT3}.
\end{proof}

\begin{remark}\label{rem:control}
	Note that if we write $\Psi=(u,v,w)$, the control system \eqref{eq:controlsystem} in differential notation becomes 
	\begin{align}
		\begin{split}\label{eq:controlproof}
			\dot u &= f(u,v) + S_0+\tilde b(w) +\sigma(w)\dot h, \\
			\dot v &= g(u,v), \\
			\dot w &= S_0+\tilde b(w) +\sigma(w)\dot h.
		\end{split}
	\end{align}
	The last equation is still autonomous and $w$ can be forced to any trajectory $\rho\in C^1\big( [0,\infty);\tW\big)$ with $\rho(0)=w(0)$. Indeed, if we choose 
	\begin{equation}\label{eq:hsimple}
	\dot h := \sigma^{-1}(\rho)\left( \dot \rho - S_0 - \tilde b(\rho)\right) \in C\big([0,\infty);\R^M\big)\subset\L^2_{\mathrm{loc}}\big([0,\infty);\R^M\big),
	\end{equation}
	then $w=\rho$ solves the third equation of \eqref{eq:controlproof}. The equations for $u$ and $v$ on the other hand can then be rewritten as
	\[
		\dot\psi(t)=F_{\dot\rho}(t,\psi(t)) \quad \text{with $\psi:=(u,v)$,}
	\]
	and therefore $\psi=\phi[(u(0),v(0)),\dot \rho]$. In other words, a suitable choice of $\dot h$ allows us to view $\psi$ as a solution to \eqref{eq:ODEphi} with a given continuous signal $S=\dot\rho$, while $w=\rho$ is its anti-derivative.
\end{remark}

\begin{remark}\label{rem:chameleon}
As a consequence of Remark \ref{rem:control} and the Support Theorem, the first two components $(X_t,Y_t)$ of the solution $\Phi_t$ of the stochastic system \eqref{eq:SDEPhi} with the \emph{fixed} periodic signal $S_0$ can (up until some time $t_0\in(0,\infty)$) approximately imitate the behavior of the solution of the deterministic system \eqref{eq:ODEphi} with \emph{any} (not necessarily periodic) signal $S\in C\big([0,\infty);\R^N\big)$ satisfying
\[
	\tW_S:=\left\{z\in\tW\,\middle|\,z+\int_0^tS(s)ds\in\tW \text{ for all $t\in[0,t_0]$}\right\}\neq\emptyset.
\]
Indeed, if $\phi_0=(x_0,y_0)\in\tU\times\tV$ is such that $\phi[(x_0,y_0),S](t)=(x,y)(t)$ exists for all $t\in[0,t_0]$, and if $\Phi_0=(x_0,y_0,z_0)$ with $z_0\in \tW_S$, \cite[Theorem 3.1]{HLT3} yields that
\[
	\P\Big(\abs{(X_t,Y_t)-(x,y)(t)}<\eps \text{ for all $t\in[0,t_0]$}\Big)>0 \quad \text{for all $\eps>0$.}
\]
Note that for $\tW=\R^N$ we have $\tW_S=\R^N$, regardless of the choice of the signal.
\end{remark}

\begin{remark}\label{rem:controlapproach}
The observation in Remark \ref{rem:control} inspires the following basic scheme for acquiring a $T$-attainable point: 
\begin{enumerate}
	\item Identify an equilibrium or a $T^*$-periodic orbit (with $T^*\notin\Q T$) of the deterministic system \eqref{eq:ODEphi} with $S\equiv0$ to which its solution is locally attracted.
	\item Define $\dot h$ in such a way that the resulting input $\dot w$ steers $\psi=(u,v)$ into the domain of attraction of this orbit. Meanwhile, $w$ must be guided to a given $z^*\in\interior(\tW)$.
	\item Let $w$ rest at $z^*$ such that from that time on $\dot w\equiv 0$ and $\psi$ converges to the attractive periodic orbit.
\end{enumerate}
This way we establish $(\phi^*,z^*)$ as a $T$-attainable point whenever $\phi^*$ lies on the orbit. The choice of $z^*\in\interior(\tW)$ turns out to be completely arbitrary for our technique and our setting (see Theorems \ref{thm:controlsimple} and \ref{thm:control} below).
\end{remark}

We will apply the strategy from Remark \ref{rem:controlapproach} in different scenarios: We start with a rather simple setting in Theorem \ref{thm:controlsimple}, where we can illustrate the general idea in a context that allows to minimize the technical aspects. The main result of this section is then presented in Theorem \ref{thm:control}, which treats a different, more nuanced and localized setting that is inspired by the Hodgkin--Huxley model. Finally, Theorem \ref{thm:controlperiodic} demonstrates how our technique can be adapted to periodic orbits of \eqref{eq:ODEphi} that are induced by periodic signals $S=S_*\not\equiv 0$. All of these theorems have in common that they do not rely on ad hoc arguments exploiting properties of the $N+L+N$-dimensional control system \eqref{eq:controlsystem} itself, but work entirely under assumptions only on the $N+L$-dimensional deterministic system \eqref{eq:ODEphi} (except for condition (C1) above).

\begin{remark}\label{rem:constantinput}
We want to make two remarks on the scope of the strategy from Remark \ref{rem:controlapproach} beyond what is explicitly dealt with in the sequel.

1.) If $c\in\R^N$, replacing $f$ with $f-c$ and $S=S_0$ with $S_0+c$ leaves the equation \eqref{eq:ODE} entirely unaltered. In \eqref{eq:SDE} this change influences only the external equation for $Z$. Since the $Z$-variable has no intrinsic meaning, one can always apply this transformation without fundamentally changing the setup. In this sense, working with properties of the system with zero-input (instead of general constant input) is hardly any restriction for practical purposes.

2.) In principle, it is possible that there is a solution to \eqref{eq:ODEphi} such that all of its components are periodic, but their individual periodicities are not the same. If they are linearly independent over $\Q$, an extension of the arguments from Section \ref{sect:lemmas} may allow to use the same strategy anyway.
\end{remark}

\subsection{Globally attractive periodic orbit under zero-input}

We want to use the approach outlined in Remark \ref{rem:controlapproach} in a relatively simple setting where the orbit is in fact globally attractive. Before we give the exact conditions, let us stress that we can handle both periodic orbits and single equilibrium points, which are included in the following as orbits with period 0.
\begin{enumerate}
	\item[\textbf{(C2)}] \textbf{Incommensurable periodic orbit:} For some $T^*\in [0,\infty)$ with $T^*=0$ or $T^* \notin \Q T$, the system \eqref{eq:ODEphi} with $S\equiv0$ has a $T^*$-periodic orbit $\cO^*\subset\interior(\tU\times\tV)$, i.e. for all $\phi_0\in\cO^*$ we have
	\[
	\phi[\phi_0,S\equiv0]([0,\infty))=\cO^*, \quad \text{and} \quad T^*=\inf\{t\in(0,\infty)\,|\,\phi[\phi_0,S\equiv0](t)=\phi_0\}.
	\]
	\item[\textbf{(C3)}] \textbf{Global attractivity of $\cO^*$:} For any $\phi_0 \in \tU\times\tV$, there is a unique global solution to \eqref{eq:ODEphi} with $S\equiv0$, and it satisfies
	\[
	\dist{\phi[\phi_0,S\equiv0](t)}{\cO^*}\xrightarrow{t\to\infty} 0.
	\]
	\item[\textbf{(C4)}] \textbf{Non-explosiveness for moderate input:} For any $\phi_0 \in \tU\times\tV$, there is a $\delta_0=\delta_0(\phi_0)>0$ such that for any $S\in C^\infty_b\big([0,\infty);\R^N\big)$ with $\norm{S}_\infty \le \delta_0$, the system \eqref{eq:ODEphi} has a unique global solution.
\end{enumerate}
We then have the following result.

\begin{thm}\label{thm:controlsimple}
Assume that (C1) - (C4) hold. Then for any $\phi^*\in\cO^*$ and $z^*\in\interior(\tW)$, the point $\Phi^*=(\phi^*,z^*)\in\interior(\tE)$ is $T$-attainable.
\end{thm}

\begin{proof}
	Let $\Phi_0=(\phi_0,z_0)\in \tE$ and write $\Psi=(\psi,w)$ as in Remark \ref{rem:control}. Since $\tW$ is convex, we can use linear interpolation and a mollifier in order to construct a smooth function $\rho\in C^\infty_b\big( [0,\infty);\tW\big)$ with $\rho(0)=z_0$ and $\norm{\dot \rho}_\infty\le\delta_0$ such that $\rho\equiv z^*$ on $[kT,\infty)$ for some $k=k(z_0,\delta_0)\in\N$. Defining $\dot h$ as in \eqref{eq:hsimple}, we can force $w$ onto $\rho$ and then $\psi=(u,v)$ has to be the solution $\phi[\phi_0,\dot \rho]$ to \eqref{eq:ODEphi}, which uniquely exists thanks to (C4). From time $kT$ on, the pertaining signal $\dot \rho$ rests at zero, so we have
	\begin{equation}\label{eq:controlsimpleproof}
			\psi(kT+t)=\phi[\psi(kT),S\equiv0](t) \quad \text{for all $t\in[0,\infty)$,}
	\end{equation}
	and hence Lemma \ref{lem:period} below will imply that $\phi^*\in \overline{\psi(\N T)}$. Since $w\equiv z^*$ on $[kT,\infty)$, this means that $\Phi^*=(\phi^*,z^*)$ is in $\overline{\Psi(\N T)}$. Consequently, $\Phi^*$ is $T$-attainable by Lemma \ref{lem:attainthesame}.
\end{proof}

\begin{remark}\label{rem:notallphi0}
In general, it is possible that (C2) and (C4) are fulfilled and the global attractivity (C3) is violated in only one point $\tilde \phi_0\in\tU\times\tV$ that is an unstable equilibrium of the system \eqref{eq:ODEphi} with $S\equiv0$ (see Example \ref{ex:TOYcontrolglobal} below). This leads to a problem in the proof of Theorem \ref{thm:controlsimple}, since we might be very unlucky and $\psi(kT)=\tilde \phi_0$. If there are $m\in\N$ and $S\in C_b^\infty\big([0,mT];\R^N\big)$ with 
\[
	\int_0^{mT}S(t)dt=0 \quad \text{and} \quad \phi[\tilde \phi_0,S](mT)\neq\tilde \phi_0,
\]
this problem can be solved by replacing $\rho$ with
\[
	\tilde \rho(t) = \begin{cases} \rho(t), & \text{if $t\in[0,\infty)\setminus [kT,(k+m)T]$,}\\ z^*+\int_0^tS(s)ds, & \text{if $t\in[kT,(k+m)T]$}, \end{cases}
\]
Then, \eqref{eq:controlsimpleproof} holds with $k$ replaced with $\tilde k = k+m$ and the rest of the proof can be finished with no further adjustments.
\end{remark}

\hspace{-6mm}\begin{minipage}{0.5\textwidth}
	\begin{example}\label{ex:TOYcontrolglobal}
		Choose $\tU=\tV=\R$ and
		\begin{align*}
		f(x,y)&=x-y-x(x^2+y^2), \\
		g(x,y)&=x+y-y(x^2+y^2)
		\end{align*}
		for all $x,y\in\R$. We see that
		\[
			[0,\infty)\ni t \mapsto \phi[(1,0),S\equiv0](t) = \begin{pmatrix} \cos t \\ \sin t \end{pmatrix}
		\]
		establishes the $2\pi$-periodic orbit $\cO^*=\mathbb S^1$
		for the corresponding system \eqref{eq:ODEphi} with $S\equiv0$. Looking at the phase diagram in Figure \ref{fig:phasen}, it is evident that the only initial condition violating (C3) is the unstable equilibrium in $\tilde \phi_0=(0,0)$, but that this can be resolved by using Remark \ref{rem:notallphi0}
		\[
		S(t) = \cos\frac{2\pi t}{mT} \quad \text{for all $t\in[0,mT]$}
		\]
		with some $m\in\N$. Obviously, (C4) is unproblematic.
	\end{example}
\end{minipage}
\hspace{6mm}\begin{minipage}{0.4\textwidth}
	\includegraphics[width=\textwidth]{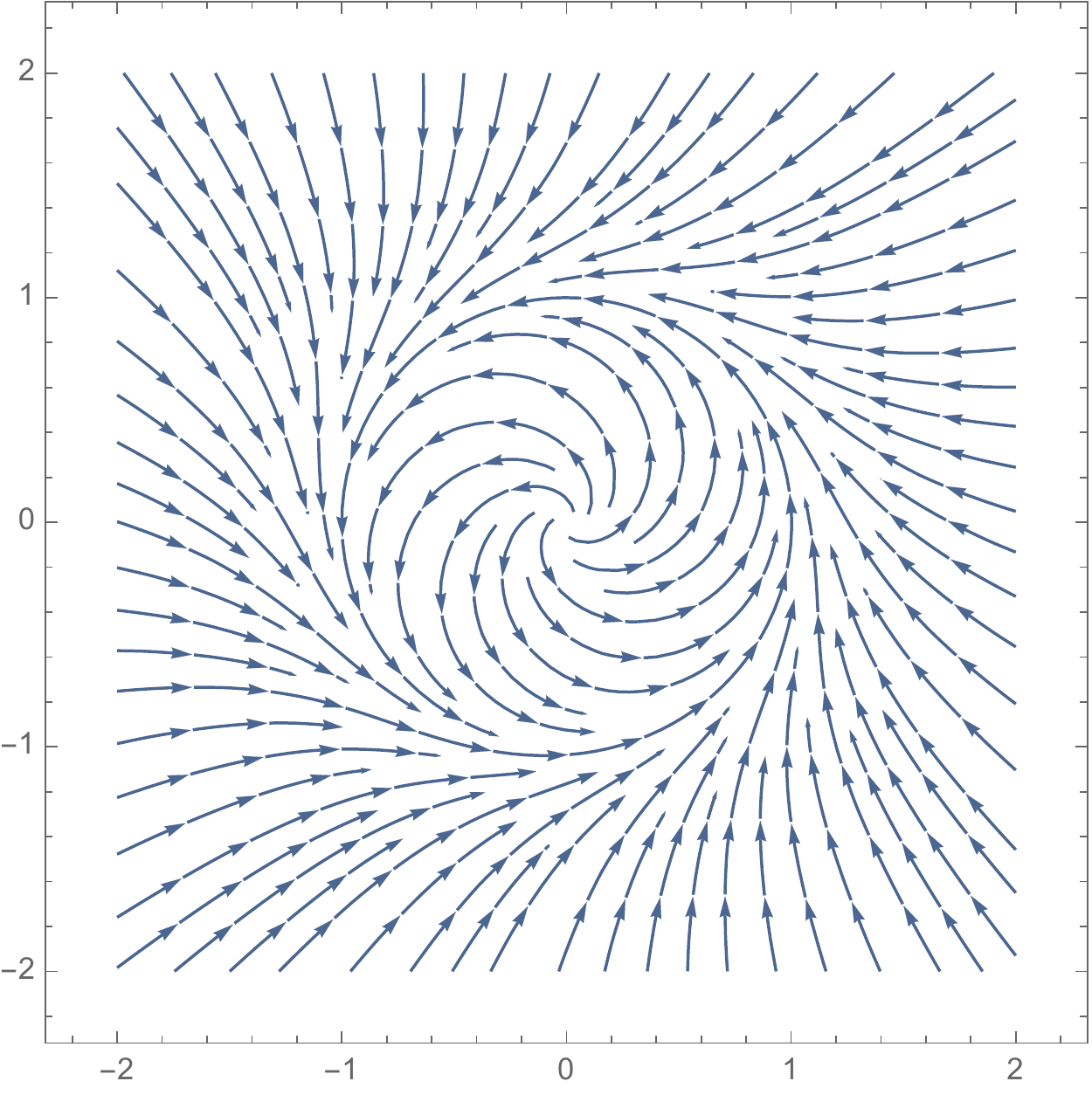}
	\captionof{figure}{}
	\label{fig:phasen}
\end{minipage}

\begin{example}\label{ex:ROTORcontrol}
	In the oscillator model from Example \ref{ex:rotors}, we may think of a situation in which the interaction potentials are bounded, while the pinning potentials pull the oscillators to the centre with a force that grows with each rotation (for example $u_1=u_2=u_3=-\phi(\abs{\,\cdot\,})$ for some non-negative and non-decreasing function $\phi$ with $\phi(0)=0$ and with a sufficiently steep slope). We do not give a rigorous proof, but it is intuitively clear that such a system will satisfy (C2) - (C4) with $\cO^*=\{0\}$, $T^*=0$.
\end{example}

\subsection{Locally attractive periodic orbit under zero-input}

The main difficulty in the technique from Remark \ref{rem:controlapproach} is that the two tasks described in the second step may obviously compete with each other. In Theorem \ref{thm:controlsimple}, we avoided this problem by simply assuming that the domain of attraction of the orbit is the entire space. In order to adapt our method to the case where the orbit is in fact only locally attractive, we subdivide this step into two steps:
\begin{enumerate}
	\item[2a.] Define $\dot h$ in such a way that the resulting input $\dot w$ steers $\psi=(u,v)$ into the domain of attraction of this orbit. 
	\item[2b.] Move $w$ to a given $z^*\in\interior(\tW)$ "gently enough" such that in the meantime $\psi$ cannot leave the domain of attraction.
\end{enumerate}
In order to realize Step 2a, we make some further assumptions on the interplay of $x$ and $y$, which are inspired by the behaviour of the Hodgkin--Huxley model. For Step 2b, we have to make sure that the system is not too sensitive with respect to slight changes of the input. These ideas are made precise in the following set of assumptions. In order to grasp the essential idea, it may be helpful on first reading to think of the case $\cO^*=\{(x^*,y^*)\}$, i.e.\ the periodic orbit is in fact just a single equilibrium point.

\begin{enumerate}
	\item[\textbf{(C3')}] \textbf{Local attractivity of $\cO^*$:} There is some $\eps^* >0$ such that for any initial value $\phi_0 \in B_{\eps^*}(\cO^*)$ there is a unique global solution to \eqref{eq:ODEphi} with $S\equiv0$, and it satisfies
	\[
	\dist{\phi[\phi_0,S\equiv0](t)}{\cO^*}\xrightarrow{t\to\infty} 0.
	\]
	\item[\textbf{(C4')}] \textbf{Signal-dependent stability of $\phi^*$:} For all $\eps>0$, there is some $\delta(\eps)\in(0,\eps]$ such that for all $S\in C^\infty_b\big([0,\infty);\R^N\big)$ with $\norm{S}_\infty < \delta(\eps)$ and all $\phi_0 \in B_{\delta(\eps)}(\cO^*)$ there is a unique global solution to \eqref{eq:ODEphi}, and it satisfies
	\[
	\phi[\phi_0,S](t) \in B_\eps(\cO^*) \quad \text{for all $t\in[0,\infty)$.}
	\]
	\item[\textbf{(C5)}] \textbf{Global attractivity in $y$ for equilibrious $x$:} There is a continuous function
	\[
		x^*\colon [0,\infty)\to \{x\in\tU\,|\,(x,y)\in\cO^* \text{ for some $y\in\tV$}\}
	\]
	such that for all $y_0\in \tV$ the initial value problem
	\[
	\dot y(t) = g(x^*(t),y(t)), \quad y(0)=y_0,
	\]
	has a unique solution $y[y_0,x^*]\colon[0,\infty)\to \tV$, and it satisfies
	\[
	\dist{(x^*(t),y[y_0,x^*](t))}{\cO^*} \xrightarrow{t\to\infty} 0.
	\]
	\item[\textbf{(C6)}] \textbf{Relocation to $\cO^*$ in $x$:} For all $\phi_0=(x_0,y_0) \in \tU \times \tV$, there are $t_1\in(0,\infty)$ and a function $\gamma\in C^1\big([0,\infty);\tU\big)$ with $\gamma(0)=x_0$ and $\gamma(t)= x^*(t)$ for all $t\in[t_1,\infty)$ (with $x^*$ from (C5)) such that the differential equation
	\begin{equation}\label{eq:tilde}
	\dot y(t) = g(\gamma(t),y(t)), \quad y(0)=y_0,
	\end{equation}
	has a unique local solution $\tilde y[y_0,\gamma]\colon[0,t_1]\to \tV$.
\end{enumerate}

\begin{remark}
Assumption (C4') is a stability condition that takes into account not only the starting point but also the signal: If we want the solution of \eqref{eq:ODEphi} to stay close to $\cO^*$, we can achieve this by starting not too far away from it, even if we slightly vary the input. If the input is in fact constantly zero, hypotheses (C3') and (C5) signify that $\cO^*$ is locally attractive, and its attraction is even global for $y$ if $x$ is already on a suitable trajectory on the projection of $\cO^*$ to $\tU$. Condition (C6) makes sure that the $x$-variable can always be relocated to such a trajectory, while the corresponding equation for $y$ still admits a local solution up to the time the $x$-variable hits $x^*$. Due to (C5), this solution can immediately be extended to a global solution. Thus, (C6) can equivalently be replaced by
\begin{enumerate}
	\item[\textbf{(C6')}] \textbf{Relocation to $\cO^*$ in $x$:} For all $\phi_0=(x_0,y_0) \in \tU \times \tV$, there are $t_1\in(0,\infty)$ and a function $S\in C\big([0,\infty);\R^N\big)$ such that there is a unique solution $(x,y)=\phi[\phi_0,S]\colon[0,\infty)\to\tU\times\tV$ to \eqref{eq:ODE}, and it satisfies $x= x^*$ on $[t_1,\infty)$.
\end{enumerate}
If (C6') is satisfied, we can simply take $\gamma$ as the corresponding trajectory of $x$ and acquire (C6). If on the other hand (C6) holds, we can deduce (C6') by extending $y$ to a global solution with (C5) and then taking $S=\dot\gamma-F(\gamma,y)$. While (C6') may be more intuitive, as it enhances the idea that one has to adjust $x$ by giving the right input, (C6) is of a form that is usually easier to verify.
\end{remark}

The set of conditions is now complete, and we can state the main result of this section.

\begin{thm}\label{thm:control}
	Let (C1), (C2), (C3'), (C4'), (C5), and (C6) hold, and let $\tW=\R^N$. Then for any $\phi^*\in\cO^*$ and $z^*\in\R^N$ the point $\Phi^*=(\phi^*,z^*)\in\interior(\tE)$ is $T$-attainable.
\end{thm}

\begin{proof} Let $\Phi_0=(x_0,y_0,z_0)\in \tE$ and write $\Psi=(u,v,w)$ as in \eqref{eq:controlproof}. As indicated above, the construction of a suitable control path is much more delicate than in Theorem \ref{thm:controlsimple}. Therefore, we divide its construction into several steps.

	\textit{Phase 1:} At first, we want to choose $\dot h$ in such a way that $u$ is forced to the trajectory $\gamma\colon[0,\infty)\to \tU$ from (C6). Therefore, we take a function $(u_\sharp,v_\sharp,w_\sharp)$ with
	\[
		u_\sharp = \gamma, \quad \dot v_\sharp = g(\gamma,v_\sharp), \quad \dot w_\sharp = \dot\gamma -f(\gamma,v_\sharp),
	\]
	and $v_\sharp(0)=y_0, w_\sharp(0)=z_0$. Here, the first equation indeed just prescribes $u_\sharp$ to coincide with $\gamma$, while (C6) secures that the second equation admits a solution in $[0,t_1]$, and the last equation is then simply solved by integrating, as the right hand side does not contain the variable $w_\sharp$. Rearranging the first equation of the control system \eqref{eq:controlproof} for $u$, we see that if we set
	\begin{align}\label{eq:hcontrol}
		\begin{split}
			\dot h :=& \sigma^{-1}(w_\sharp)\left( \dot w_\sharp - S_0 - \tilde b(w_\sharp)\right) \\
			=& \sigma^{-1}(w_\sharp)\left( \dot u_\sharp-  f(u_\sharp,v_\sharp) - S_0 - \tilde b(w_\sharp)\right) \in C\big([0,\infty);\R^M\big),
		\end{split}
	\end{align}
	the function $(u,v,w):=(u_\sharp,v_\sharp,w_\sharp)$ satisfies \eqref{eq:controlproof} for all $t\in[0,t_1]$. Thus we have constructed a control path up until the time $t_1$, from which on the first variable $u$ follows the curve $x^*$.
	
	\textit{Phase 2:} Since $u_\sharp=\gamma=x^*$ after time $t_1$, the stability condition (C5) implies that $v_\sharp$ can be extended to a global solution with
	\[
	\dist{(u_\sharp(t),v_\sharp(t))}{\cO^*}\xrightarrow{t\to\infty}0,
	\]
	which allows us to set
	\[
	t_2:=\inf\left\{t>t_1\,\middle|\, (u_\sharp(t),v_\sharp(t)) \in B_{\frac12\delta(\eps^*)}(\cO^*) \right\} \in [t_1,\infty)
	\]
	with $\eps^*$ from (C3') and $\delta(\eps^*)$ chosen according to (C4'). Thus, by the time $t_2$ we have driven $(u_\sharp,v_\sharp)$ into $B_{\delta(\eps^*)}(\cO^*)$.
	
	\textit{Phase 3:} Up to time $t_2$, we have constructed a useful candidate for $\dot h$ and the corresponding solution $(u,v,w)=(u_\sharp,v_\sharp,w_\sharp)$ to \eqref{eq:controlproof}. Our next step will be to change the definition \eqref{eq:hcontrol} of $\dot h$ beyond time $t_2$ and extend this solution accordingly. Our aim is to drive $w$ to $z^*$ while keeping $(u,v)$ close to $\cO^*$. Since $\dot w$ can be thought of as an alternate signal that is fed into the equation, assumption (C4') makes sure that we can achieve this, as long as we move $w$ sufficiently slowly. Linear interpolation and classical mollification techniques allow us to take some $\rho\in C^\infty_b\big([t_2,\infty);\R^N\big)$ with $\rho(t_2)=w_\sharp(t_2)$ and $\norm{\dot\rho}_\infty<\delta(\eps^*)$ such that for some $k\in\N$ with $kT>t_2$ we have $\rho\equiv z^*$ on $[kT,\infty)$. On $[t_2,\infty)$, we change the definition of $\dot h$ to
	\begin{equation}\label{eq:hcontrol2}
	\dot h := \sigma^{-1}(\rho)\left( \dot\rho - S_0 -\tilde b(\rho)\right) \in C\big([t_2,\infty);\R^M\big).
	\end{equation}
	With this choice, we can obviously extend $w$ by taking it equal to $\rho$ beyond time $t_2$, and thus $\psi=(u,v)$ has to obey
	\begin{equation}\label{eq:phase3}
		\dot \psi(t) = F_{\dot \rho}(t,\psi(t)) \quad \text{for all $t\in[t_2,\infty).$}
	\end{equation}
	Since $\dot\rho$ now plays the role of a signal whose absolute value is bounded by $\delta(\eps^*)$, assumption (C4') yields that we can indeed extend $\psi$ correspondingly and $\psi(kT)$ will still be in $B_{\eps^*}(\cO^*)$. Furthermore, $w\equiv z^*$ and $\dot w\equiv0$ on $[kT,\infty)$.
		
	\textit{Phase 4:} After time $kT$, the equation \eqref{eq:phase3} turns into $\dot \psi= F(\psi)$ and therefore
	\[
		\psi(kT+t)=\phi[\psi(kT),S\equiv0](t) \quad \text{for all $t\in[0,\infty)$.}
	\]
	Since $\psi(kT) \in B_{\eps^*}(\cO^*)$, we can apply Lemma \ref{lem:period} and the proof is completed in the same way as the proof of Theorem \ref{thm:controlsimple}.
\end{proof}

\begin{remark}
	The route we take in this proof is inspired by the one used in the proof of \cite[Proposition 2.5]{HLT3}, where the authors treat the stochastic Hodgkin--Huxley system. They do not take into account general periodic orbits but only a constant equilibrium. For the special case in which the autonomous equation for $Z$ is of one-dimensional Cox-Ingersoll-Ross type, they can in fact abandon the need for $\sigma$ to be well-defined on the entire space $\R$. This would make the first two phases of our proof more complicated, as one has to make sure that $w$ does not leave the state space by turning negative (this is why we assume $\tW=\R^N$). The case that $Z$ is of Ornstein--Uhlenbeck type -- of which our model is a generalization -- is treated mostly in the same way as here, albeit only for the dimension $M=N=1$ and for constant diffusion coefficient. However, there is one key difference in the third phase of the control (which corresponds to "Part (V)" of the proof of \cite[Proposition 2.5]{HLT3}): In the mentioned article, the control is defined such that $u$ is forced to a trajectory that moves it very slightly away from the equilibrium point $(x^*,y^*)$. If this is done slowly enough, $v$ is assumed to stay close to $y^*$, and $w$ approaches $z^*$ similarly as in our third phase. No rigorous prove is given that the deterministic Hodgkin--Huxley system actually shows this kind of behaviour. We think that in general it is a more intuitive approach to force not $u$ but $w$ to a specific trajectory, since we can interpret this as feeding a suitable signal into \eqref{eq:ODE} that does not let $(x,y)$ escape from the domain of attraction of the equilibrium. Condition (C4') is the key to making this possible. Note though that for the Hodgkin--Huxley model, we still rely on numerical simulations and merely give a rough intuitive argument that (C4') holds (see Example \ref{ex:HHcontrol} below). However, for a simpler example, the conditions of Theorem \ref{thm:control} can be checked rigorously (for a detailed discussion, we refer the reader to \cite[Example 2.29]{ICHDiss}).
\end{remark}

In the following examples, we want to discuss properties of certain systems like \eqref{eq:ODE} that allow for the application of Theorem \ref{thm:control}. The basic condition (C1) only concerns $\sigma$, and we simply assume that it holds.

\begin{example}\label{ex:HHcontrol}
	We return to the stochastic Hodgkin--Huxley system, the stability of which has mainly been studied numerically. There is an open neighborhood $\cC$ of 0 such that every constant input $S\equiv c \in \cC$ is injectively and continuously mapped to an equilibrium
	\[
	(x^*_c,y^*_c)=\left(x^*_c,\frac{\alpha_1(x^*_c)}{(\alpha_1+\beta_1)(x^*_c)},\frac{\alpha_2(x^*_c)}{(\alpha_2+\beta_2)(x^*_c)},\frac{\alpha_3(x^*_c)}{(\alpha_3+\beta_3)(x^*_c)} \right)\in \R \times (0,1)^3,
	\]
	which -- as simulations suggest -- is stable and locally attractive (see \cite[pages 533 and 548]{HLT3}), i.e.\ for all $\eps>0$ there is a $\delta=\delta(c,\eps)>0$ such that for all $\phi_0 \in B_\delta(x^*_c,y^*_c)$ we have
	\begin{equation}\label{eq:HHcontrol1}
	(x,y)[x_0,y_0,S\equiv c](t) \in B_\eps(x^*_c,y^*_c) \quad \text{for all $t\in[0,\infty)$},
	\end{equation}
	and there is some $\eps^*>0$ such that for all $\phi_0 \in B_{\eps^*}(x^*_c,y^*_c)$ we actually even have
	\[
	\phi[(x_0,y_0),S\equiv c](t) \xrightarrow{t\to\infty} (x^*_c,y^*_c).
	\]
	Therefore, (C2) and (C3') are taken for granted, and we write $\phi^*:=(x^*_0,y^*_0)$. Simulations show that if started slightly beside the equilibrium, the solution is immediately pulled in its direction, suggesting that one can in fact take $\delta=\eps$ in \eqref{eq:HHcontrol1}. This is also the reason why stability in the sense of \eqref{eq:HHcontrol1} should remain valid when also slightly and slowly changing the input over time as we do with the function $\rho$ in the third control phase in the proof of Theorem \ref{thm:control}: A small change of the input $c$ simply corresponds to a small change of the equilibrium $(x^*_c,y^*_c)$ towards which the trajectory is headed. This is essentially (C4'). The conditions (C5) and (C6) (with $x^*\equiv x_0^*$) follow easily from the fact that for a given trajectory of the membrane potential $x$, variation of constants yields that we can write
	\begin{equation*}\label{eq:HHinternal}
	y_i(t)=y_i(0)e^{-\int_0^t(\alpha_i+\beta_i)(x(s))ds}+\int_0^t\alpha_i(x(s))e^{-\int_s^t(\alpha_i+\beta_i)(x(r))dr}ds
	\end{equation*}
	for every $i\in\{1,2,3\}$ and $t\in[0,\infty)$.
\end{example}

\begin{example}\label{ex:toycontrolN=1}
We consider the system from Example \ref{ex:toy} for $N=1$ and the choice of $g$ that is given in \eqref{eq:toyg2}. For $S\equiv0$, the points $\phi^*=(\pm1,1,\ldots,1)$ fulfill (C2), (C3'), (C5) and (C6). The key to seeing this lies in considering $x$ and $y$ separately and using the classical Lyapunov method for each variable in a way that works (locally) uniformly in the other variable (see \cite[Example 2.29]{ICHDiss} for a detailed calculation for a similar example). Condition (C4') can be checked with essentially the same arguments as well (ibidem).
\end{example}

\begin{example}\label{ex:TOYcontrolorbit}
We take the functions $f$ and $g$ from Example \ref{ex:TOYcontrolglobal} and outside of $B_2(0)\subset\R^2$ we change their dependence on $x$ arbitrarily. Then $\cO^*=\mathbb S^1$ is still a periodic orbit of the corresponding system \eqref{eq:ODEphi} with $S\equiv0$, but now only local attractivity is guaranteed in $x$. However, (C3'), (C4') and (C5) (with $x^*=\cos(\cdot)$) can be checked via similar arguments as in Example \ref{ex:toycontrolN=1}. Note that this separation of the variables automatically solves the problem we had with the unstable equilibrium at the origin in Example \ref{ex:TOYcontrolglobal}.
\end{example}

\subsection{Globally attractive periodic solution under nonzero-input}

The scheme outlined in Remark \ref{rem:controlapproach} may be refined in the following way:
\begin{enumerate}
		\item Find a $T_*$-periodic signal $S_*$ (with $T_*\notin\Q T$) with mean zero over one period such that the corresponding deterministic system \eqref{eq:ODEphi} has a $T^*$-periodic solution (with $T^*\in\Q T_*$) that is locally attractive.
		\item Define $\dot h$ in such a way that the resulting input $\dot w$ steers $\psi=(u,v)$ into the domain of attraction of this solution. At the same time, $w$ must be guided to a given $z^*\in\tW$.
		\item Once $w$ reaches $z^*$, choose $\dot h$ such that $\dot w= S_*$. Then $w$ will keep returning to $z^*$ while $\psi$ approaches the attractive periodic solution.
\end{enumerate}
Under certain conditions on the interplay of the different occurring periodicities, we can then establish $(\phi^*,z^*)$ as a $T$-attainable point whenever $\phi^*$ is a suitable value of the periodic solution.

\begin{enumerate}
	\item[\textbf{(C2*)}] \textbf{Incommensurable periodic solution:} There is a non-constant $T_*$-periodic signal $S_*\in C([0,\infty);\R^N)$ with
	\[
	T_*\in (0,\infty) \setminus \Q T \quad \text{and} \quad \int_0^{T_*}S_*(t)dt=0
	\]
	such that for some $T^*\in\Q T_*$ and $\phi_* \in \interior(\tU\times\tV)$, the unique solution
	\[
	\phi^*(t):=\phi[\phi_*,S_*](t), \quad t\in[0,\infty),
	\]
	to the corresponding system \eqref{eq:ODEphi} is $T^*$-periodic and $\phi^*([0,T^*])\subset\interior(\tU\times\tV)$.
	\item[\textbf{(C3*)}] \textbf{Global attractivity of $\phi^*$:} For any $\phi_0 \in \tU\times\tV$, there is a unique global solution to \eqref{eq:ODEphi} with $S=S_*$, and it satisfies
	\[
	\abs{\phi[\phi_0,S_*](t)-\phi^*(t)}\xrightarrow{t\to\infty} 0.
	\]
	\item[\textbf{(C4*)}] \textbf{Non-explosiveness for moderately varied input:} For any $\phi_0 \in \tU\times\tV$, there is a $\delta_0=\delta_0(\phi_0)>0$ such that for any $S\in C\big([0,\infty);\R^N\big)$ with $\norm{S-S_*}_\infty \le \delta_0$, the system \eqref{eq:ODEphi} has a unique global solution.
\end{enumerate}
We can then prove the following result.

\begin{thm}\label{thm:controlperiodic}
	Let $\tW=\R^N$ and assume that (C1) and (C2*) - (C4*) hold. Then for any $t^* \in [0,T^*)$ and $z^*\in\R^N$ the point $\Phi^*=(\phi^*(t^*),z^*)\in\interior(\tE)$ is $T$-attainable.
\end{thm}

\begin{proof}
	Note that for any $t^*\in (0,T^*)$, the conditions (C2*) - (C4*) are also satisfied with $S_*(t^*+\,\cdot\,)$, $\phi^*(t^*)$, and $\phi^*(t^*+\,\cdot\,)$ in place of $S_*$, $\phi_*$, and $\phi^*$. This proof can therefore be reduced to the case $\phi^*(t^*)=\phi^*(0)=\phi_*$.
	
	Let $\Phi_0=(\phi_0,z_0)\in \tE$, $z^*\in\R^N$ and write $\Psi=(\psi,w)$ as in Remark \ref{rem:control}. Using linear interpolation, we can construct a smooth function $\kappa\colon[0,\infty)\to\tW$ with $\norm{\kappa}_\infty\le\delta_0$ such that $\kappa\equiv 0$ on $[k_0T_*,\infty)$ for some $k_0=k_0(z_0,\delta_0)\in\N$ and
	\[
	\int_0^{k_0T_*}\kappa(s)ds=z^*-z_0.	
	\]
	Then, the continuously differentiable function
	\[
	\rho\colon[0,\infty)\to\R^N, \quad t \mapsto z_0+\int_0^t\left(S_*(s)+\kappa(s)\right)ds
	\]
	fulfills not only
	\begin{equation}\label{eq:wreturns}
	\rho(0)=z_0 \quad \text{and} \quad \rho((k_0+k)T_*)=z^* \quad \text{for all $k\in\N$,}
	\end{equation}
	but also
	\begin{equation}\label{eq:rhocontrolsimple}
	\norm{\dot \rho-S_*}_\infty\le\delta_0 \quad \text{and} \quad \dot \rho=S_* \quad \text{on $[k_0T_*,\infty)$.}
	\end{equation}
	Now, we can once again define $\dot h$ as in \eqref{eq:hsimple} in order to force $w$ onto $\rho$ and then $\psi=(u,v)$ has to be the solution $\phi[\phi_0,\dot \rho]$ to \eqref{eq:ODEphi}, which uniquely exists thanks to (C4*). By \eqref{eq:rhocontrolsimple}, we have
	\[
	\abs{\psi(k_0T_*+t)-\phi^*(t)}=\abs{\phi[\psi(k_0T_*),S_*](t)-\phi^*(t)} \xrightarrow{t\to\infty}0,
	\]
	and since $\phi^*$ has a period that is a rational multiple of $T_*$, we can combine this with \eqref{eq:wreturns} and conclude that $\Phi^*\in \overline{\Psi(\N T_*)}$. Since $T_*\notin\Q T$, an argument similar to the one in Lemma \ref{lem:period} (but simpler) shows that also $\Phi^*\in \overline{\Psi(\N T)}$.
\end{proof}

\begin{example}
Let $T\in\Q$, $\tU=\tV=\tW=\R$ and
\[
	f(x,y)=-\big(1+y^2\big)x, \quad g(x,y)=x-y  \quad \text{for all $x,y\in\R$.}
\]
The corresponding system \eqref{eq:ODEphi} fulfills the conditions of Theorem \ref{thm:controlsimple} with $\cO^*=\{(0,0)\}\subset\R^2$, and hence any point on the line $\{(0,0)\}\times\R$ is $T$-attainable. If we set
\[
	S_*(t)=\cos t , \quad \phi^*(t)=\frac12\begin{pmatrix}\cos t +\sin t \\ \sin t \end{pmatrix} \quad \text{for all $t\in[0,\infty)$,}
\]
the corresponding system \eqref{eq:ODEphi} with $S=S_*$ fulfills the assumptions of Theorem \ref{thm:controlperiodic}, showing that any point on the tube $\phi^*([0,2\pi])\times\R$ is also $T$-attainable.
\end{example}

\begin{example}\label{ex:TOYnonzeroinput}
We revisit Example \ref{ex:toy} once again, this time for $N=2$ and for $g$ as in \eqref{eq:toyg1}. If $S\equiv0$, each point in $(\mathbb S^1 \cup \{0\})\times\{0\}\subset \R^{2+L}$ is an equilibrium of the system \eqref{eq:ODEphi}, and no further periodic orbits exist. Unfortunately, none of these equilibria are attractive in the sense of (C3) or (C3'), so Theorems \ref{thm:controlsimple} and \ref{thm:control} provide no help at all in finding a $T$-attainable point. However, for $T\in\Q$ the assumptions of Theorem \ref{thm:controlperiodic} are satisfied with
\[
S_*(t)=\begin{pmatrix}\cos t \\ \sin t \end{pmatrix}, \quad \phi^*(t)=\begin{pmatrix}\sin t \\ -\cos t \\ 0\end{pmatrix} \quad \text{for all $t\in[0,\infty)$,}
\]
showing that indeed any point in $\phi^*([0,2\pi])\times\R^2=\mathbb S^1\times\{0\}\times\R^2\subset \R^{2+L+2}$ is $T$-attainable. The influence of the signal $S_*$ can be interpreted as a constant rotation of the mexican hat with the exact speed needed to turn a curve of unstable equilibria into a globally asymptotically stable periodic solution.
\end{example}

\subsection{Two Lemmas}\label{sect:lemmas}

\begin{lem}\label{lem:period}
	Let (C2) hold and let $\phi_0\in \tU\times\tV$. Further assume one of the following conditions:
	\begin{enumerate}
		\item[(a)] (C3) holds.
		\item[(b)] (C3') holds and $\phi_0\in B_{\eps^*}(\cO^*)$.
	\end{enumerate}
	Then $\cO^*\subset\overline{\phi[\phi_0,S\equiv0](\N T)}$.
\end{lem}

\begin{proof}
	Let $\phi^*\in\cO^*$. If $T^*=0$ and hence the orbit contains only a single equilibrium point, this lemma's claim is trivial. Therefore we assume $T^*\in(0,\infty)\setminus \Q T$, write $\phi:=\phi[\phi_0,S\equiv0]$ and let $\eps>0$. We want to show that there is some $m\in\N$ such that $\phi(mT)\in B_{2\eps}(\phi^*)$. In both cases (a) and (b), $\phi$ converges to $\cO^*$, so there is some $t_0=t_0(\eps)\in(0,\infty)$ with $\phi([t_0,\infty))\subset B_\eps(\cO^*)$. Since $\cO^*=\phi[\phi^*,S\equiv0]([0,T^*])$ is compact, $F$ is bounded on $B_\eps(\cO^*)$. As $\dot\phi=F(\phi)$, this means that $\phi|_{[t_0,\infty)}$ is Lipschitz continuous with a Lipschitz constant $L(\eps)\in(0,\infty)$. Let
	\[
	\cN:=\left\{\phi^*+\xi\,\middle|\,\xi\in\R^{N+L}, \, \xi\cdot F(\phi^*)=0\right\}\subset\R^{N+L}
	\]
	denote the hyperplane that is orthogonal to the orbit $\cO^*$ in the point $\phi^*$. This is indeed a well-defined hyperplane, since $T^*>0$ and therefore $F(\phi^*)\neq0$. If $T^*_n$ denotes the $n$-th time after $t_0$ at which $\phi$ passes through $\cN\cap B_\eps(\phi^*)$, \cite[Lemma 8.5.3 and its proof]{MeyerOffin} together with the attraction property (C3) or (C3') yield that $T^*_n$ is well-defined for sufficiently small $\eps>0$ and that we have
	\[
	T^*_n-T^*_{n-1} \xrightarrow{n\to\infty}T^*.
	\]
	Then Lemma \ref{lem:Kronecker} below provides $m,n\in\N$ with
	\[
	|T^*_n-mT|<\frac{\eps}{L(\eps)}
	\] 
	and hence we have
	\[
	\abs{\phi(mT)-\phi^*}\le L(\eps)\abs{mT-T^*_n}+\abs{\phi(T^*_n)-\phi^*}<2\eps,
	\]
	which completes the proof.
\end{proof}

\begin{lem}\label{lem:Kronecker}
	Let $T\in(0,\infty)$ and $T^*\in (0,\infty)\setminus \Q T$, and let $(T^*_n)_{n\in\N}\subset \R$ with
	\begin{equation}\label{eq:lemKronecker}
	T^*_n-T^*_{n-1} \xrightarrow{n\to\infty}T^*.
	\end{equation}
	Then for any $\eps>0$, there are $m,n\in\mathbb N$ such that
	\[
	|T^*_n-mT|<\eps.
	\]
\end{lem}

\begin{proof}
	Thanks to a strengthened version of Kronecker's Approximation Theorem (compare \cite[Chapter 7, Exercise 3]{Kronecker}), we know that for each $\eps>0$ there is an $N(\eps)\in\N$ such that for all $t\in[0,\infty)$ there are $n(t,\eps) \in \{1,\ldots, N(\eps)\}$ and $m(t,\eps)\in\N$ with the property
	\[
	|t+n(t,\eps)T^*-m(t,\eps)T|<\eps.
	\]
	Now let $\eps>0$. Thanks to the assumption \eqref{eq:lemKronecker} there is an $n_0\in\N$ such that
	\[
	T^*_{n_0}>0 \quad \text{and}\quad \abs{(T^*_k-T^*_{k-1})-T^*}<\frac{\eps}{2N\big(\frac{\eps}{2}\big)} \quad \text{for all $k\ge n_0$.}	
	\]
	Then for all $n\ge n_0$ and $m\in\N$ we have
	\begin{align*}
		\abs{T^*_n-mT} &= \abs{\sum_{k=n_0+1}^n\big((T^*_{k}-T^*_{k-1})-T^*\big)+T^*_{n_0}+(n-n_0)T^*-mT}\\
		&\le (n-n_0)\frac{\eps}{2N\big(\frac{\eps}{2}\big)}+\abs{T^*_{n_0}+(n-n_0)T^*-mT}
	\end{align*}
	and choosing $n=n_0+n\big(T^*_{n_0},\frac{\eps}{2}\big)$ and $m=m\big(T^*_{n_0},\frac{\eps}{2}\big)$ finishes the proof.
\end{proof}

\section{Local H{\"o}rmander condition}\label{sect:lwh}

In this section, we will present conditions on the external equation and on the way $X$ and $Y$ interact through the functions $f$ and $g$ that allow to prove condition (III) of Theorem \ref{thm:HLT} via the  so-called \emph{local H{\"o}rmander condition} (see Definition \ref{def:lwh} and Lemma \ref{lem:hoermanderIII} below).

\begin{notation}
	Whenever it is convenient, we will denote elements of $\R^{1+N+L+N}$ by
	\begin{equation}\label{eq:xi}
	\xi=(\xi_0,\ldots,\xi_{N+L+N}) :=(t,x_1,\ldots,x_N,y_1,\ldots,y_L,z_1,\ldots,z_N) =(t,x,y,z)=(t,\phi,z)=(t,\Phi)
	\end{equation}
	and use the abbreviation
	\[
	\hat b(t,z):=S_0(t)+\tilde b(z) \quad \text{for all $(t,z)\in[0,\infty)\times\tW$,}
	\]
	where $\tilde b$ is as it was defined in \eqref{eq:stratodrift}. For any $m,n\in\N$, we will also use the notation
	\[
	A_{\cdot,k}:=\begin{pmatrix} A_{1,k} \\ \vdots \\ A_{n,k}	\end{pmatrix} \in \R^n \quad \text{for all $k\in\{1,\ldots,m\}$}
	\]
	for the $k$-th column of a matrix $A \in\R^{n\times m}$, and we will write $0_n$ for the zero vector in $\R^n$.
\end{notation} 


For vector fields $V_1,V_2\in C^1(D;\R^n)$ on some $D\subset\R^n$, $n\in\N$, the \emph{Lie bracket} is defined as the vector field
\[
[V_1,V_2]:=J_{V_2}V_1 - J_{V_1} V_2 \in C(D;\R^n),
\]
where $J_{V_i}$ denotes the respective Jacobian matrix. Since this operation is bilinear, we can and will view the space $C^\infty(D;\R^n)$ as an algebra that has the Lie bracket as its multiplication. For any $\cC\subset C^\infty(D;\R^n)$, we write $\langle\cC\rangle$ for the subalgebra generated by $\cC$. For any subalgebra $\cL\subset C^\infty(D;\R^n)$ containing $\cC$, we write $[\cC]_\cL$ for the ideal in $\cL$ that is generated by $\cC$.

The idea behind the Lie bracket is that by combining motions in the directions $V_1$ and $V_2$ one can effectively approximate motion in the directions $[V_1,V_2]$, $[V_1,[V_1,V_2]]$, \ldots and ultimately in any direction in $\langle\{V_1,V_2\}\rangle$. If $V_1$ and $V_2$ occur as the drift or as a column of the diffusion matrix of a Stratonovich SDE, this can be exploited in order to determine along which directions its solution can evolve locally. A detailed heuristic explanation of this idea can be found in \cite[Section 2]{HairerMalli}.

If we want to apply this reasoning to our time-inhomogeneous setting, we have to consider the homogeneous $(1+N+L+N)$-dimensional time-space process $(t,\Phi_t)_{t\in[0,\infty)}$. It solves the Stratonovich SDE
\[
d(t,\Phi_t)=V_0 (t,\Phi_t)dt+\sum_{k=1}^M V_k(t,\Phi_t)\circ dW^{(k)}_t
\]
with the vector fields $V_k \colon [0,\infty)\times\tE\to\R^{1+N+L+N}$ defined by
\begin{equation}\label{eq:Vk}
V_0(t,\Phi):= \begin{pmatrix} 1 \\ \tilde B (t,\Phi)\end{pmatrix}, \quad V_k(t,\Phi):= \begin{pmatrix} 0 \\ \Sigma_{\cdot,k} (\Phi)  \end{pmatrix} \quad \text{for all $k \in\{1,\ldots,M\}$,}
\end{equation}
with $B$ and $\Sigma$ as in \eqref{eq:B} and \eqref{eq:Sigma}. The following Definition \ref{def:lwh} formalizes the idea that, locally, the process $(\Phi_t)_{t\in [0,\infty)}$ should be able to move in any direction. Context-specific variants are classically used in order to prove the existence of transition densities, usually incorporating tools from the Malliavin calculus (see for example \cite{HairerMalli} and the references therein).

\begin{definition}\label{def:lwh}
	We set
	\[
	\cC :=  \{V_1,\ldots,V_M\}, \quad \cL :=  \langle\{V_0,\ldots,V_M\}\rangle, \quad\text{and}\quad \cL^*:=	[\cC]_\cL.
	\]
	If there is some $\Phi\in\interior(\tE)$ such that
	\begin{equation}\label{eq:LH}
	\spann \{V(t,\Phi)\,|\, V\in \cL^* \} \simeq \R^{N+L+N} \quad \text{for all $t\in [0,\infty)$},
	\end{equation}
	we say that the \emph{local H{\"o}rmander condition holds at $\Phi$}.
\end{definition}


\begin{lem}\label{lem:hoermanderIII}
	The local H{\"o}rmander condition at $\Phi^*\in\interior(\tE)$ implies condition (III) of Theorem \ref{thm:HLT}.
\end{lem}

\begin{proof}
This follows from \cite[Theorems 1 and 2]{HLT2} and Remark \ref{rem:HLT}.
\end{proof}

For general diffusions, there is very little one can say about non-trivial sufficient conditions for the local H{\"o}rmander condition. One usually depends on ad hoc arguments that exploit the particular shape of the drift and the diffusion matrix. For systems of the special type \eqref{eq:SDE} however, there are two fairly general scenarios that we can treat without confining us to entirely specific examples. Both are based on the idea that external noise (in the form of the $Z$-variable) is explicitly imported only into the $X$-variable, while the coefficient functions $f$ and $g$ have to transport and distribute its influence suitably among all of the $X$- and $Y$-variables.

\begin{enumerate}
	\item[\textbf{(I)}] \textbf{Star shape: }Every component depends on the $X$-variables in such a way that sufficient amounts of noise are able to spread from $Z$ via $X$ to the rest of the system.  
	\item[\textbf{(II)}] \textbf{Cascade structure: }There is a chain of components such that with each step exactly one more of them is directly influenced by the previous one, and this chain ultimately runs through the entire system.
\end{enumerate}

Of course, in both of these scenarios we also have to make sure that $X$ and $Z$ are not coupled in a degenerate way. However, since the random terms in their respective equations coincide, their interaction is basically coded in the difference of their drift coefficients, which is again simply given by the function $f$. It therefore seems natural that, as long as $\sigma$ is nice enough for $Z$ not to be degenerate itself, it should be possible for us to find conditions solely on $f$ and $g$ that are sufficient for the local H{\"o}rmander condition.

The star shape setting is inspired by the Hodgkin-Huxley system and its variants (see Section 4 of \cite{HLT2}, compare Example \ref{ex:HHlwh} below), while the cascade structure setting is inspired by a model of interacting neurons that is studied in Section 5.3 of \cite{Cascade} (see Remark \ref{rem:H1H2}). In this section, both scenarios will be treated under more general assumptions that transcend the scope of these specific models.

Before we proceed in treating each of these situations separately in the Subsections \ref{sect:star} and \ref{sect:cascade} below, we will first collect some preparatory observations and calculations.


\begin{lem}\label{lem:liebasic}
	Let $\cU\subset\tE$ open and assume that for $(t,\Phi)\in[0,\infty)\times\bar\cU$ the mappings
	\[
	(t,z) \mapsto A_i(t,z) \in \R^N, \quad	\phi \mapsto W_i(\phi) \in \R^{N+L}, \quad (t,z) \mapsto a_i(t,z) \in \R,
	\]
	are continuously differentiable for $i\in\{1,2\}$. Then for all $(t,\Phi)\in[0,\infty)\times\bar\cU$ we have
	\begin{align}
	\text{1.}&\left[\begin{pmatrix} A_1(t,\cdot) \\ 0_L \\ A_1(t,\cdot) \end{pmatrix}, \begin{pmatrix} A_2(t,\cdot) \\ 0_L \\ A_2(t,\cdot) \end{pmatrix}\right]\!(\Phi) = \begin{pmatrix} [A_1(t,\cdot),A_2(t,\cdot)](z) \\ 0_L \\ [A_1(t,\cdot),A_2(t,\cdot)](z) \end{pmatrix}, \notag \\
	\text{2.}&\left[ a_1(t,\cdot) \begin{pmatrix} W_1 \\ 0_N \end{pmatrix}, a_2(t,\cdot) \begin{pmatrix} W_2 \\ 0_N \end{pmatrix}\right]\!(\Phi) = a_1(t,z)a_2(t,z)\begin{pmatrix} [W_1,W_2](\phi) \\ 0_N \end{pmatrix}, \notag \\
	\text{3.}&\left[a_1(t,\cdot) \begin{pmatrix} W_1 \\ 0_N \end{pmatrix},\begin{pmatrix} A_1(t,\cdot) \\ 0_L \\ A_1(t,\cdot) \end{pmatrix}\right]\!(\Phi) = \displaystyle{-\sum_{j=1}^N} A_1^{(j)}(t,z)\Bigg( a_1(t,z) \begin{pmatrix} \d_{x_j}W_1(\phi) \\ 0_N \end{pmatrix} 
	+ \d_{z_j}a_1(t,z) \begin{pmatrix} W_1(\phi) \\ 0_N \end{pmatrix}\Bigg). \notag
	\end{align}
\end{lem}

\begin{proof}
	All of these formulas follow immediately from the definition of the Lie bracket by straight forward calculations.
\end{proof}

\begin{remark}\label{rem:lietime}
	Let us make some comments on the role of time, i.e.\ the $(0)$-components of the occurring vector fields on the one hand and the derivatives with respect to $t=\xi_0$ on the other hand.
	
	1.) First, note that for all $k\in\{0,\ldots,M\}$ the $(0)$-component of $V_k$ is constant. Hence, for any vector field $W\in C^1\big([0,\infty)\times\tE;\R^{1+N+L+N}\big)$ we have
	\[
	[V_k,W]^{(0)}=\sum_{i=0}^{N+L+N} V_k^{(i)}\d_{\xi_i}W^{(0)},
	\]
	which vanishes everywhere whenever $W^{(0)}$ is constant as well. In particular, this is the case for $W\in\{V_0,\ldots,V_M\}$. Consequently, any vector field that is contained in $\cL^*$ will have a vanishing $(0)$-component. Hence, $N+L+N$ is the maximum dimension that can possibly be achieved in \eqref{eq:LH} despite the fact that $V_0,\ldots,V_M$ are $(N+L+N+1)$-dimensional.
	
	2.) Let $V,W\in C^1\big([0,\infty)\times\tE;\R^{1+N+L+N}\big)$. As
	\[
	[V,W]= V^{(0)}\d_tW- W^{(0)}\d_tV +\sum_{i=1}^{N+L+N}\left( V^{(i)}\d_{\xi_i}W- W^{(i)}\d_{\xi_i}V\right),
	\]
	any possible influence of time derivatives of $W$ is killed by a vanishing $V^{(0)}$ and vice versa. The first part of this remark then yields that no time derivatives occur explicitly in
	\begin{equation}\label{eq:lienotime}
	[V,W]=\sum_{i=1}^{N+L+N}\big( V^{(i)}\d_{\xi_i}W- W^{(i)}\d_{\xi_i}V\big) \quad \text{for all $V,W \in \cL^*$,}
	\end{equation}
	while
	\begin{equation}\label{eq:lietime}
	[V_0,W]=\d_t W+\sum_{i=1}^{N+L+N}\big( V_0^{(i)}\d_{\xi_i}W- W^{(i)}\d_{\xi_i}V_0\big) \quad \text{for all $W \in \cL^*$.}
	\end{equation}
	As $\d_t V_k$ vanishes for all $k\in\{1,\ldots,M\}$, time derivatives cannot occur before taking Lie brackets at least twice in the construction of $\cL^*$.
\end{remark}

The following notational convention is of utmost importance for understanding the rest of this section.

\begin{notation}\label{not:lie}
	1.) As noted under 1.)\ in Remark \ref{rem:lietime}, the only relevant vector field to feature a non-vanishing $(0)$-component is $V_0$ -- which is not contained in $\cL^*$. In order to simplify our notation, in the sequel \textbf{we will therefore systematically omit this component}. More precisely, we identify every vector field
	\begin{equation}\label{eq:notlie1}
	W\colon[0,\infty)\times \tE \to \R^{1+N+L+N}, \quad (t,\Phi) \mapsto \begin{pmatrix} 0 \\ W^{(1)}(t,\Phi) \\ \vdots \\ W^{(N+L+N)}(t,\Phi) \end{pmatrix},
	\end{equation}
	with the collection of vector fields on $\tE$ given by
	\begin{equation}\label{eq:notlie2}
	W(t,\cdot)\colon \tE \to \R^{N+L+N}, \quad \Phi \mapsto \begin{pmatrix} W^{(1)}(t,\Phi) \\ \vdots \\ W^{(N+L+N)}(t,\Phi) \end{pmatrix}, \quad \text{for all $t \in [0,\infty)$.}
	\end{equation}
	Either of these objects will simply be denoted by $W$. Let us stress that both directions of this identification are actively used: on the one hand, we tacitly omit vanishing zero-components when a vector field is given as in \eqref{eq:notlie1}, but on the other hand, we also think of each vector field that is given in terms of \eqref{eq:notlie2} as one on $[0,\infty)\times\tE$ with a vanishing zero-component.
	
	2.) If $V$ and $W$ are of this type, so is their Lie bracket $[V,W]$ (confer part 1.)\ of Remark \ref{rem:lietime}). Thus, it is natural (and consistent) to identify
	\begin{equation}\label{eq:notlie3}
	[V,W](t,\Phi)=[V(t,\cdot),W(t,\cdot)](\Phi) \in \R^{N+L+N}
	\end{equation}
	for all $(t,\Phi)\in[0,\infty)\times\tE$.
	
	3.) As indicated above, $V_0$ is the only relevant vector field not of the type in \eqref{eq:notlie1} and \eqref{eq:notlie2}, but we only need it when taking Lie brackets as in \eqref{eq:lietime}. Using the notation we just introduced, \eqref{eq:lietime} can be rewritten as
	\begin{equation}\label{eq:lietime+}
	[V_0,W]=\d_t W+\left[\begin{pmatrix} \hat b\\ 0_L \\ \hat b \end{pmatrix} + \begin{pmatrix} F \\ 0_N	\end{pmatrix},W\right].
	\end{equation}
	Note that this formula is completely consistent: $[V_0,W]^{(0)}$ and $\d_t W^{(0)}$ vanish again and are therefore omitted, while the Lie bracket on the right hand side is to be understood as the Lie bracket of vector fields on $\tE$ (which carry an additional dependence on time) and thus has only $N+L+N$ components to begin with.
	
	4.) Note also that with this notational convention there is no difference between $V_k$ and $\Sigma_{\cdot,k}$ for $k \in\{1,\ldots,M\}$.
	
	5.) Let us further illustrate this notation by rewriting the formulas from Lemma \ref{lem:liebasic} accordingly. With the identification of \eqref{eq:notlie1} and \eqref{eq:notlie2} and with the convention \eqref{eq:notlie3}, they read
	\begin{align*}
		\text{1. }&\left[\begin{pmatrix} A_1 \\ 0_L \\ A_1 \end{pmatrix}, \begin{pmatrix} A_2 \\ 0_L \\ A_2 \end{pmatrix}\right]\!(\xi) = \begin{pmatrix} [A_1(t,\cdot),A_2(t,\cdot)](z) \\ 0_L \\ [A_1(t,\cdot),A_2(t,\cdot)](z) \end{pmatrix},\\
		\text{2. }&\left[ a_1 \begin{pmatrix} W_1 \\ 0_N \end{pmatrix}, a_2 \begin{pmatrix} W_2 \\ 0_N \end{pmatrix}\right]\!(\xi) = a_1(t,z)a_2(t,z)\begin{pmatrix} [W_1,W_2](\phi) \\ 0_N \end{pmatrix},\\
		\text{3. }&\left[a_1 \begin{pmatrix} W_1 \\ 0_N \end{pmatrix},\begin{pmatrix} A_1 \\ 0_L \\ A_1 \end{pmatrix}\right]\!(\xi) = \displaystyle{-\sum_{j=1}^N} A_1^{(j)}(t,z)\Bigg( a_1(t,z) \begin{pmatrix} \d_{x_j}W_1(\phi) \\ 0_N \end{pmatrix} 
		\d_{z_j}a_1(t,z) \begin{pmatrix} W_1(\phi) \\ 0_N \end{pmatrix}\Bigg)
	\end{align*}
	for all $\xi=(t,\phi,z)\in[0,\infty)\times\bar\cU$. 
\end{notation}


Using Notation \ref{not:lie} and Lemma \ref{lem:liebasic}, we will now see which vector fields can be constructed by taking Lie brackets with $V_0,\ldots,V_M$. The first formula in Lemma \ref{lem:liebasic} yields
\begin{equation}\label{eq:lievolatility}
\cL^*\supset\langle\{V_1,\ldots,V_M\}\rangle=\left\{ \begin{pmatrix} w \\ 0_L \\ w \end{pmatrix} \,\middle|\, w \in \langle\{\sigma_{\cdot,1},\ldots,\sigma_{\cdot,M}\}\rangle \right\},
\end{equation}
i.e.\ without incorporating the drift, we can span at most $N$ dimensions. 

\begin{definition}
We say that \emph{$\sigma$ is non-degenerate in $z\in\tW$} if
\begin{equation}\label{eq:sigmanon-degenerate}
	\spann\{ w(z)\,|\, w \in \langle\{\sigma_{\cdot,1},\ldots,\sigma_{\cdot,M}\}\rangle \}= \R^N.
\end{equation}
\end{definition}

If, for example, $\sigma(z)$ is surjective (as in condition (C1) in Section \ref{sect:control}), it has to have $N$ linearly independent columns, and hence we get non-degeneracy without even taking Lie brackets.

Let $(x,y,z)\in\tE$ and assume that $\sigma$ is non-degenerate in $z$. In order to span the remaining $N+L$ dimensions, we will always start by taking the Lie bracket of the drift $V_0$ with some column $V_k$ of the diffusion matrix, $k\in\{1,\ldots,M\}$. Using \eqref{eq:lietime+} and Lemma \ref{lem:liebasic}, we see that
\begin{align}\label{eq:iacascadelemma1}
	\begin{split}
		[V_0,V_k](\xi) 
		&=\begin{pmatrix} [\hat b(t,\cdot),\sigma_{\cdot,k}](z) \\ 0_L \\ [\hat b(t,\cdot),\sigma_{\cdot,k}](z) \end{pmatrix} - \sum_{i=1}^N \sigma_{i,k}(z) \begin{pmatrix} \d_{x_i}F(x,y) \\ 0_N	\end{pmatrix}
	\end{split}
\end{align}
and hence
\begin{equation}\label{eq:iastarlemma1}
[V_k,V_0](\xi)=-[V_0,V_k](\xi)=\begin{pmatrix} [\sigma_{\cdot,k},\hat b(t,\cdot)](z) \\ 0_L \\ [\sigma_{\cdot,k},\hat b(t,\cdot)](z) \end{pmatrix} + \sum_{i=1}^N \sigma_{i,k}(z) \begin{pmatrix} \d_{x_i}F(x,y) \\ 0_N	\end{pmatrix}.
\end{equation}
In order to get a better idea of what we are dealing with here, let us see what happens when we take an extra Lie bracket with $V_0$ or some column $V_l$ of the diffusion matrix, $l\in\{1,\ldots, M\}$. Using \eqref{eq:lietime+}, Lemma \ref{lem:liebasic}, and basic properties of the Lie bracket again, we see that 

\begin{align}\label{eq:iacascadelemma2}
\begin{split}
[V_0,[V_0,V_k]](\xi)&=\d_t [V_0,V_k](\xi) + \left[\begin{pmatrix} \hat b\\	0_L \\ \hat b	\end{pmatrix}+\begin{pmatrix} F \\ 0_N	\end{pmatrix}, [V_0,V_k] \right]\!(\xi) \\
&=\begin{pmatrix} [\hat b(t,\cdot),[\hat b(t,\cdot),\sigma_{\cdot,k}]](z)+\d_t[\hat b(t,\cdot),\sigma_{\cdot,k}](z) \\ 0_L \\ [\hat b(t,\cdot),[\hat b(t,\cdot),\sigma_{\cdot,k}]](z)+\d_t[\hat b(t,\cdot),\sigma_{\cdot,k}](z) \end{pmatrix} \\ 
& \quad- \sum_{i=1}^N \Bigg( [\hat b(t,\cdot),\sigma_{\cdot,k}]^{(i)}(z) \begin{pmatrix} \d_{x_i}F(x,y) \\ 0_N	\end{pmatrix} + \sigma_{i,k}(z)\begin{pmatrix} [F,\d_{x_i}F](x,y) \\ 0_N	\end{pmatrix} \\
& \;\qquad\qquad + \sum_{j=1}^N \hat b^{(j)}(t,z)\left( \sigma_{i,k}(z) \begin{pmatrix} \d_{x_j}\d_{x_i}F(x,y) \\ 0_N	\end{pmatrix}+ \d_{z_j}\sigma_{i,k}(z) \begin{pmatrix} \d_{x_i}F(x,y) \\ 0_N	\end{pmatrix}\right) \Bigg).
\end{split}
\end{align}
Similarly, we can calculate
\begin{align}\label{eq:iastarlemma2}
	\begin{split}
		[V_l,[V_k,V_0]](\xi)
		&= \begin{pmatrix} \zeta_{k,l}(t,z) \\ 0_L \\ \zeta_{k,l}(t,z) \end{pmatrix} + \sum_{i,j=1}^N  \sigma_{j,l}(z)\Bigg(\sigma_{i,k}(z) \begin{pmatrix} \d_{x_j}\d_{x_i}F(x,y) \\ 0_N	\end{pmatrix} 
		+	\d_{z_j}\sigma_{i,k}(z) \begin{pmatrix} \d_{x_i}F(x,y) \\ 0_N	\end{pmatrix}\Bigg),
	\end{split}
\end{align}
where
\begin{align}\label{eq:zeta}
	\begin{split}
		\zeta_{k_1}(t,z)&:=[\sigma_{\cdot,k_1},\hat b(t,\cdot)](z), \\
		\zeta_{k_1,\ldots,k_n}(t,z)&:=[\sigma_{\cdot,k_n},\zeta_{k_1,\ldots,k_{n-1}}(t,\cdot)](z) \quad \text{for all $n\ge2$,}
	\end{split}
\end{align}
with any $k_1,k_2,\ldots \in\{1,\ldots,M\}$.

Having thus acquired a basic idea of the typical structure of higher iterations
\[
	[V_0, [ \ldots, [V_0,V_k]]], [V_{k_n}, [ \ldots, [V_{k_1},V_0]]] \in\cL^*,
\]
we formulate the following two Lemmas that will help us treat these in detail later. Their proofs are straight forward, following the same line of arguments as in the above calculations of $[V_0,[V_0,V_k]]$ and $[V_l,[V_k,V_0]]$.


\begin{lem}\label{lem:liecascade}
	If there is an open $\cU\subset\tE$ such that $W\in C^1\big(\tE,\R^{1+N+L+N}\big)$ can locally be written as 
	\[
	W(\xi)=\begin{pmatrix} A(t,z)\\ 0_L \\ A(t,z) \end{pmatrix} - \sum_{i=1}^n a_i(t,z) \begin{pmatrix} W_i(x,y) \\ 0_N \end{pmatrix} \quad \text{for all $\xi\in[0,\infty)\times\bar\cU$,}
	\]
	for suitable $n\in\N$ and smooth functions $A$, $a_i$ and $W_i$, then
	\begin{align*}
		[V_0,W](\xi)
		= & \begin{pmatrix} [\hat b(t,\cdot),A(t,\cdot)](z)+\d_t A(t,z)\\ 0_L \\ [\hat b(t,\cdot),A(t,\cdot)](z)+\d_t A(t,z) \end{pmatrix} 
		- \sum_{i=1}^n \d_ta_i(t,z) \begin{pmatrix} W_i(x,y) \\ 0_N \end{pmatrix} \\
		&- \sum_{j=1}^N A^{(j)}(t,z)\begin{pmatrix} \d_{x_j}F(x,y) \\ 0_N \end{pmatrix}
		- \sum_{i=1}^na_i(t,z) \begin{pmatrix} [F,W_i](x,y) \\ 0_N \end{pmatrix} \\
		&- \sum_{i=1}^n\sum_{j=1}^N \hat b^{(j)}(t,z)\left( a_i(t,z) \begin{pmatrix} \d_{x_j}W_i(x,y)\\ 0_N \end{pmatrix} +\d_{z_j}a_i(t,z) \begin{pmatrix}W_i(x,y)\\ 0_N \end{pmatrix} \right)
	\end{align*}
	for all $\xi\in[0,\infty)\times\bar\cU$.
\end{lem}

\begin{lem}\label{lem:liestar}
	If there is an open $\cU\subset\tE$ such that $W\in C^1\big(\tE,\R^{1+N+L+N}\big)$ can locally be written as 
	\[
	W(\xi)=\begin{pmatrix} A(t,z)\\ 0_L \\ A(t,z) \end{pmatrix} + \sum_{i=1}^n a_i(z) \begin{pmatrix} W_i(x,y) \\ 0_N \end{pmatrix} \quad \text{for all $\xi\in[0,\infty)\times\cU$.}
	\]
	for suitable $n\in\N$ and smooth functions $A$, $a_i$ and $W_i$, then
	\begin{align*}
		[V_k,W](\xi)
		= \begin{pmatrix} [\sigma_{\cdot,k},A(t,\cdot)](z) \\ 0_L \\ [\sigma_{\cdot,k},A(t,\cdot)](z) \end{pmatrix}  &+ \sum_{i=1}^n\sum_{j=1}^N \sigma_{j,k}(z)\left( a_i(z)\begin{pmatrix} \d_{x_j}W_i(x,y) \\ 0_N \end{pmatrix} +\d_{z_j}a_i(z)\begin{pmatrix} W_i(x,y) \\ 0_N \end{pmatrix}\right)
	\end{align*}
	for all $\xi\in[0,\infty)\times\bar\cU$ and any $k\in\{1,\ldots,M\}$.
\end{lem}


\subsection{Star shape}\label{sect:star}

The aim of this section is to prove the following theorem and two other variants of it that also include the case $N\neq M$ (see Theorem \ref{thm:star} and Corollary \ref{cor:starcst} below).

\begin{thm}\label{cor:stardiag}
	Let $\Phi=(x,y,z)\in\interior(\tE)$, $M=N$ and assume that $\sigma(z)$ is an invertible diagonal matrix. If for some $k_1,\ldots,k_{N+L}\in\{1,\ldots,N\}$ the vectors
	\begin{equation}\label{eq:corstardiag}
	\d_{x_{k_1}}\cdots\d_{x_{k_n}}F(x,y), \quad n\in\{1,\ldots,N+L\},
	\end{equation}
	are linearly independent, then the local H{\"o}rmander condition holds at $\Phi$.
\end{thm}

Our proof strategy is to span all dimensions with vector fields that are constructed in the following fashion: Let $l\in\N$ and $\kappa=(k_1,\ldots,k_l) \in\{1,\ldots,M\}^l$ and define
\[
L_{\kappa,1}:=[V_{k_1},V_0] \quad \text{and} \quad L_{\kappa,n}:=[V_{k_n},L_{\kappa,n-1}] \quad \text{for all $n\in\{2,\ldots,l\}$}.
\]
When this recursive procedure is finished, we have acquired the vector field $L_{\kappa,l}\in\cL^*$. As this is the result of successively taking Lie brackets with vector fields that follow the "path" $\kappa=(k_1,\ldots,k_l)$ through the columns of the diffusion matrix, we will use the notation $L_\kappa:=L_{\kappa,l}$. Our goal is to find verifiable criteria under which for some sequence $\kappa_1,\ldots,\kappa_{N+L}$ of such paths (of possibly different lengths) and for some 
\[
W_1,\ldots,W_N \in \langle\{V_1,\ldots,V_M\}\rangle,
\]
the span of the vector fields
\[
W_1,\ldots,W_N,L_{\kappa_1}, \ldots, L_{\kappa_{N+L}} \in\cL^*,
\]
evaluated at a suitable point $\xi\in[0,\infty)\times\tE$ has dimension $N+L+N$. To this end, we will have to calculate $L_\kappa$ for a general path
\[
\kappa=(k_1,\ldots,k_{l(\kappa)})\in \bigcup_{l\in\N}\{1,\ldots,M\}^l,
\]
where we write $l(\kappa)$ for the path length of $\kappa$. Thanks to Lemma \ref{lem:liestar}, these calculations are not too hard. Recall the definition of the functions $\zeta_{k_1,\ldots,k_n}$ that was given in \eqref{eq:zeta}.

\begin{lem}\label{lem:liestar2}
	For any $\kappa=(k_1,\ldots,k_l)\in\{1,\ldots,M\}^l$, $n\in\{1,\ldots,l\}$, and $\xi\in[0,\infty)\times\tE$, we have
	\begin{align}\label{eq:L}
		\begin{split}
			L_{\kappa,n}(\xi)=\begin{pmatrix} \zeta_{k_1,\ldots,k_n}(t,z) \\ 0_L \\ \zeta_{k_1,\ldots,k_n}(t,z) \end{pmatrix}
			&+ \sum_{i_1,\ldots,i_n=1}^N \sigma_{i_1,k_1}(z)\cdots\sigma_{i_n,k_n}(z)\begin{pmatrix} \d_{x_{i_1}}\cdots\d_{x_{i_n}} F(x,y) \\ 0_N \end{pmatrix} \\
			&+ \sum_{\alpha \in \N_0^N, 1\le \abs{\alpha}_1 \le n-1} p_{\kappa,n,\alpha}(z) \begin{pmatrix} \d_x^\alpha F(x,y) \\ 0_N \end{pmatrix},
		\end{split}
	\end{align}
	where each coefficient function $p_{\kappa,n,\alpha}$ is a polynomial expression of the terms
	\[
	\d_z^\beta \sigma_{i,k_j} \quad \text{with } \beta\in\N_0^N, \, \abs{\beta}_1\le\abs{\alpha}_1, \, i\in\{1,\ldots,N\}, \, j\in\{1,\ldots,n\},
	\]
	where $\abs{\alpha}_1:=\alpha_1+\ldots+\alpha_N$ for all $\alpha\in\N_0^N$.
\end{lem}


\begin{proof}
	Fix some $\kappa=(k_1,\ldots,k_l)\in\{1,\ldots,M\}^l$. We want to prove this Lemma by induction, so at first we notice that the cases $n=1$ and $n=2$ were already checked in \eqref{eq:iastarlemma1} and \eqref{eq:iastarlemma2}. Assume now that this Lemma's claim holds for some $n\in\{1,\ldots,l-1\}$. Then Lemma \ref{lem:liestar} yields
	\begin{align*}
		L_{\kappa,n+1}(\xi)&=\begin{pmatrix} \zeta_{k_1,\ldots,k_{n+1}}(t,z) \\ 0_L \\ \zeta_{k_1,\ldots,k_{n+1}}(t,z) \end{pmatrix}
		+ \sum_{i_1,\ldots,i_{n+1}=1}^N \sigma_{i_1,k_1}(z)\cdots\sigma_{i_{n+1},k_{n+1}}(z)\begin{pmatrix} \d_{x_{i_1}}\cdots\d_{x_{i_{n+1}}} F(x,y) \\ 0_N \end{pmatrix} \\
		&\hspace{5mm}+ \sum_{i_1,\ldots,i_{n+1}=1}^N \sigma_{i_{n+1},k_{n+1}}(z)\d_{z_{i_{n+1}}}[\sigma_{i_1,k_1}(z)\cdots\sigma_{i_n,k_n}(z)]\begin{pmatrix} \d_{x_{i_1}}\cdots\d_{x_{i_n}} F(x,y) \\ 0_N \end{pmatrix}\\
		&\hspace{5mm}+ \sum_{\alpha \in \N_0^N, 1\le \abs{\alpha}_1 \le n-1}\sum_{i=1}^N \sigma_{i,k_{n+1}}(z)p_{\kappa,n,\alpha}(z)\begin{pmatrix} \d_{x_i}\d_x^\alpha F(x,y) \\ 0_N \end{pmatrix} \\
		&\hspace{5mm}+ \sum_{\alpha \in \N_0^N, 1\le \abs{\alpha}_1 \le n-1}\sum_{i=1}^N  \sigma_{i,k_{n+1}}(z)\d_{z_i}p_{\kappa,n,\alpha}(z)\begin{pmatrix} \d_x^\alpha F(x,y) \\ 0_N \end{pmatrix},
	\end{align*}
	where the summands in the first line are already exactly of the shape we are aiming for. All of the other terms correspond to derivatives of $F$ that are of the order $n$ at most. By the induction hypothesis, each $p_{\kappa,n,\alpha}$ is a polynomial expression of the terms $\d_z^\beta \sigma_{i,k_j}$ with $\beta\in\N_0^N$, $\abs{\beta}_1\le\abs{\alpha}_1$, $i\in\{1,\ldots,N\}$, $j\in\{1,\ldots,n\}$. Trivially, the same is then true for any partial derivative $\d_{z_i}p_{\kappa,n,\alpha}$ and hence all of the respective coefficients in the above formula are polynomial expressions of $\d_z^\beta \sigma_{i,k_j}$ with $i\in\{1,\ldots,N\}$, $j\in\{1,\ldots,n+1\}$ and $\beta\in\N_0^N$ with $\abs{\beta}_1$ bounded by the order of the respective derivative of $F$. Thus, $L_{\kappa,n+1}(\xi)$ is indeed of the desired form, which completes the proof.
\end{proof}

\begin{remark}\label{rem:coeff}
	Having a closer look at the formula in the proof of Lemma \ref{lem:liestar2}, we note that the coefficients $p_{\kappa,n,\alpha}$ can be calculated by the following scheme:
	\begin{enumerate}
		\item For all $n\in\N$, define
		\[
		p_{\kappa,n,0_N}:=0
		\]
		and
		\[
		p_{\kappa,n,\alpha}:=\prod_{j=1}^n \sigma_{i_j,k_j}
		\]
		for all $\alpha\in\N_0^N$ with $\abs{\alpha}_1=n$ and $\d_x^\alpha=\d_{x_{i_1}}\cdots\d_{x_{i_n}}$. 
		\item Now we can recursively calculate
		\[
		p_{\kappa,n,\alpha}=\sum_{i=1}^N\sigma_{i,k_{n}}\big(\d_{z_i} p_{\kappa,n-1,\alpha} + p_{\kappa,n-1,\alpha-e_i}\cdot 1_{\alpha_i\neq 0} \big)
		\]
		for all $n\ge 2$ and $\alpha\in\N_0^N$ with $\abs{\alpha}_1\in\{1,\ldots,n-1\}$, where $e_i$ denotes the $i$-th canonical unit vector in $\R^N$.
	\end{enumerate}
\end{remark}

Let us now consider a finite sequence of paths
\[
\kappa_1,\ldots,\kappa_{N+L} \in \bigcup_{l\in\N}\{1,\ldots,M\}^l
\]
and the corresponding sequence of vector fields $L_{\kappa_1},\ldots,L_{\kappa_{N+L}} \in \cL^*$ as defined in \eqref{eq:L}. Let $\xi=(t,x,y,z)\in[0,\infty)\times\tE$ and assume that $\sigma$ is non-degenerate in $z$, i.e.\ there are
\[
w_1,\ldots,w_N\in\langle\{\sigma_{\cdot,1},\ldots,\sigma_{\cdot,M}\}\rangle
\]
such that $w_1(z),\ldots,w_N(z)\in\R^N$ are linearly independent. Setting
\[
W_1:=\begin{pmatrix}w_1\\0_L \\w_1\end{pmatrix},\ldots,W_N:=\begin{pmatrix}w_N\\0_L \\w_N\end{pmatrix}\in\langle\{V_1,\ldots,V_M\}\rangle,
\]
we want to find a sufficient condition for
\[
W_1(z),\ldots,W_N(z),L_{\kappa_1}(\xi),\ldots,L_{\kappa_{N+L}}(\xi)
\]
to be linearly independent. We first note that for each $\kappa_n=(k_{n,1},\ldots,k_{n,l(\kappa_n)})$ the term
\[
\zeta_{k_{n,1},\ldots,k_{n,l(\kappa_n)}}(t,z) \in \R^N
\]
can of course be expressed as a linear combination of the linearly independent vectors $w_1(z),\ldots,w_N(z)\in\R^N$, and consequently the first summand
\[
\begin{pmatrix} \zeta_{k_{n,1},\ldots,k_{n,l(\kappa_n)}}(t,z) \\ 0_L \\ \zeta_{k_{n,1},\ldots,k_{n,l(\kappa_n)}}(t,z) \end{pmatrix}\in\R^{N+L+N}
\]
of $L_{\kappa_n}(\xi)$ is a linear combination of $W_1(z),\ldots,W_N(z)\in\R^{N+L+N}$. All of its other summands are zero in their last $N$ components. Therefore, the vectors
\[
W_1(z),\ldots,W_N(z), L_{\kappa_1}(\xi),\ldots,L_{\kappa_{N+L}}(\xi)\in\R^{N+L+N}
\]
are linearly independent if and only if
\[
\bar L_{\kappa_1}(\xi),\ldots,\bar L_{\kappa_{N+L}}(\xi)\in\R^{N+L}
\]
are, where
\begin{align}\label{eq:Lbar}
	\begin{split}
		\bar L_{\kappa_n}(\xi)&= \sum_{i_1,\ldots,i_{l(\kappa_n)}=1}^N \sigma_{i_1,k_{n,1}}(z)\cdots\sigma_{i_{l(\kappa_n)},k_{n,l(\kappa_n)}}(z)\d_{x_{i_1}}\cdots\d_{ x_{i_{l(\kappa_n)}}}F(x,y) \\
		&\hspace{5mm}+ \sum_{\alpha \in \N_0^N, 1\le \abs{\alpha} \le l(\kappa_n)-1} p_{\kappa_n,l(\kappa_n),\alpha}(z)\d_x^\alpha F(x,y)
	\end{split}
\end{align}
for every $n\in\{1,\ldots,N+L\}$. Note that these vector fields in fact no longer depend on time, which is why we can simply write
\[
\bar L_{\kappa_n}(x,y,z)=\bar L_{\kappa_n}(\xi).
\]
The following theorem collects the insight we have gained so far. Note that the signal (the only time-dependent object present) is not part of any of the occurring vector fields, which is why uniformity in time (as required for the local H{\"o}rmander condition) is entirely unproblematic.

\begin{thm}\label{thm:star}
	Let $\Phi=(x,y,z)\in\interior(\tE)$ and assume that the following two conditions hold.
	\begin{enumerate}
		\item[(i)] $\sigma$ is non-degenerate in $z$.
		\item[(ii)] There are
		\[
		\kappa_1,\ldots,\kappa_{N+L} \in \bigcup_{l\in\N}\{1,\ldots,M\}^l
		\]
		such that
		\[
		\bar L_{\kappa_1}(\Phi), \ldots, \bar L_{\kappa_{N+L}}(\Phi) \quad \text{are linearly independent.}
		\]
	\end{enumerate}
	Then the local H{\"o}rmander condition holds at $\Phi$.
\end{thm}

\begin{proof}
	As we have reasoned above, the assumptions of this theorem imply that there are $W_1,\ldots,W_N\in\langle\{V_1,\ldots,V_M\}\rangle\subset\cL^*$ such that the span of
	\[
	W_1(z),\ldots,W_N(z),L_{\kappa_1}(t,\Phi), \ldots, L_{\kappa_{N+L}}(t,\Phi)
	\]
	has the dimension $N+L+N$ for all $t\in[0,\infty)$. Since by construction $L_{\kappa_1}, \ldots, L_{\kappa_{N+L}}$ are also in $\cL^*$, it follows that the local H{\"o}rmander condition holds at $\Phi$.
\end{proof}

In general, the assumptions of Theorem \ref{thm:star} can turn out to be still very hard to verify, since the vector fields in \eqref{eq:Lbar} have a rather complicated structure. However, we can present certain interesting situations in which they can be radically simplified.

\begin{cor}\label{cor:starcst}
	Let $\Phi=(x,y,z)\in\interior(\tE)$ and $\kappa_1,\ldots,\kappa_{N+L} \in \bigcup_{l\in\N}\{1,\ldots,M\}^l$.
	\begin{enumerate}
		\item[(a)]
		Assume that in a neighborhood of $z$, the diffusion matrix $\sigma$ is constant and its only value is a surjective matrix. If the vectors
		\[
		\sum_{i_1,\ldots,i_{l(\kappa_n)}=1}^N \sigma_{i_1,k_{n,1}}\cdots\sigma_{i_{l(\kappa_n)},k_{n,l(\kappa_n)}}\d_{x_{i_1}}\cdots\d_{ x_{i_{l(\kappa_n)}}}F(x,y), \quad n \in\{1,\ldots,N+L\},
		\]
		are linearly independent, then the local H{\"o}rmander condition holds at $\Phi$.
		\item[(b)]
		Let $M=N$ and assume that in a neighborhood of $z$, the diffusion matrix $\sigma$ is constant and that its only value is an invertible diagonal matrix. If the vectors
		\[
		\d_{x_{k_{n,1}}}\cdots\d_{ x_{k_{n,l(\kappa_n)}}}F(x,y), \quad n \in\{1,\ldots,N+L\},
		\]
		are linearly independent, then the local H{\"o}rmander condition holds at $\Phi$.
	\end{enumerate}
\end{cor}

\begin{proof}
	Since (b) follows from (a), we only have to prove (a). As the diffusion matrix takes only one value $\sigma\in\R^{M\times N}$, condition (i) of Theorem \ref{thm:star} is in fact equivalent to surjectivity of this matrix. Furthermore, it follows from Remark \ref{rem:coeff} that for constant $\sigma$ all of the coefficients $p_{\kappa_n,l(\kappa_n),\alpha}$ with $1\le\abs{\alpha}_1\le l(\kappa_n)-1$ vanish. Therefore, $\bar L_{\kappa_1},\ldots,\bar L_{\kappa_{N+L}}$ can be simplified to the respective expressions given in this Corollary. Their linear independence yields condition (ii) of Theorem \ref{thm:star}, which can then be applied to complete the proof.
\end{proof}

With Theorem \ref{thm:star}, we can also handle the

\begin{proof}[Proof of Theorem \ref{cor:stardiag}]
	Once again, we want to apply Theorem \ref{thm:star}. By invertibility of $\sigma(z)$, all of its columns are linearly independent, and hence condition (i) is trivially fulfilled.
	
	In order to check condition (ii), we set
	\begin{equation}\label{eq:starpath}
	\kappa_n:=(k_1,\ldots,k_n) \quad \text{for every $n\in\{1,\ldots,N+L\}$.}
	\end{equation}
	The vector field $\bar L_{\kappa_n}$ can be calculated by the formula \eqref{eq:Lbar}, and since $\sigma(z)$ is a diagonal matrix, the first sum is reduced to just the one summand
	\[
	\sigma_{k_1,k_1}(z)\cdots\sigma_{k_n,k_n}(z)\d_{x_{k_1}}\cdots\d_{x_{k_n}}F(x,y),
	\]
	while the second sum is reduced to
	\[
	\sum_{l=1}^{n-1} p_{\kappa_n,n,\alpha(l)}(z) \d_{x_{k_1}}\cdots\d_{x_{k_l}} F(x,y)
	\]
	with $\alpha(l):=e_{k_1}+\ldots+e_{k_l}$, as follows immediately from Remark \ref{rem:coeff}. Invertibility of $\sigma(z)$ yields that $\sigma_{k_1,k_1}(z)\cdots\sigma_{k_n,k_n}(z)$ can never be zero, so none of the leading coefficients of $\bar L_{\kappa_n}(\Phi)$ vanish. Hence, starting with $n=1$, we can successively add multiples of $\bar L_{\kappa_n}(\Phi)$ to $\bar L_{\kappa_{n+1}}(\Phi)$ in order to eliminate any summand in $\bar L_{\kappa_{n+1}}(\Phi)$ with derivatives of the order less than $n+1$. The resulting vectors are clearly linearly independent if and only if $\bar L_{\kappa_1}(\Phi), \ldots,\bar L_{\kappa_{N+L}}(\Phi)$ are. This leads to the conclusion that linear independence of
	\[
	\sigma_{k_1,k_1}(z)\cdots\sigma_{k_n,k_n}(z)\d_{x_{k_1}}\cdots\d_{x_{k_n}}F(x,y), \quad n\in\{1,\ldots,N+L\},
	\]
	is sufficient for condition (ii) of Theorem \ref{thm:star}. As leaving out scaling factors has no effect on linear independence, the proof is completed.
\end{proof}

\begin{example}\label{ex:HHlwh}
	For $N=M=1$ and locally positive $\sigma$, Theorem \ref{cor:stardiag} contains Theorem 3 of \cite{HLT2} as a special case. In particular, it covers the Hodgkin--Huxley model, as is calculated in detail in Section 5.3 of \cite{HLT2}.
\end{example}

\begin{example}\label{ex:TOYlwhstar}
	In continuation of Example \ref{ex:TOYnonzeroinput}, we consider the toy model from Example \ref{ex:toy} for $N=2$, $L=1$ and $g$ as in \eqref{eq:toyg1}. In Example \ref{ex:TOYnonzeroinput}, we showed in particular that (for $T\in\Q$) the point $\Phi=(1,0,0,z)$ is $T$-attainable for any $z\in\R^2$. Straight forward calculations show that Theorem \ref{cor:stardiag} with $(k_1,k_2,k_3)=(1,2,2)$ yields the local H{\"o}rmander condition in $\Phi$ when $\sigma(z)\in\R^{2\times2}$ is an invertible diagonal matrix. If $M\neq N$ and for instance
	\[
		\sigma\equiv\begin{pmatrix} 1 & 1 & 2\\ 0 & 1 & 1\end{pmatrix} \quad \text{in some open set $\cU\subset\R^2$,}
	\]
	we can check the local H{\"o}rmander condition in $\Phi$ by taking $\kappa_1=(1)$, $\kappa_2=(1,2)$ and $\kappa_3=(2,2)$ and applying Corollary \ref{cor:starcst}. Since, as noted in Remark \ref{rem:lyapunov}, there is also a Lyapunov function in this setting, this means that for either choice of $\sigma$, all of the assumptions of Theorem \ref{thm:HLT} are satisfied and the solution to the corresponding SDE \eqref{eq:SDE} is positive Harris recurrent (even exponentially ergodic).
\end{example}

\subsection{Cascade structure}\label{sect:cascade}

For this section, we will stick to the case $M=N=1$. Since $\sigma$ is now $\R$-valued, we have
\begin{equation}\label{eq:sigmapos}
\tB:=\{ z\in\tW \,|\, \text{$\sigma$ is non-degenerate in $z\in\tW$} \}=\{ z\in\tW \,|\, \sigma(z)\neq 0 \}\subset\tW\subset \R.
\end{equation}
We tacitly assume that $\tB$ is non-empty, since otherwise our model would be devoid of any randomness. Moreover, we fix an open set $\tA\subset\tU\times\tV$ and write
\[
\tE^*:=\tA\times\tB\subset \tE\subset\R^{1+L+1}.
\]
Let us first state this section's main result.

\begin{thm}\label{thm:cascade}
	Consider the following conditions.
	\begin{enumerate}
		\item[\emph{\textbf{(H1)}}] In $\tA$ we have $\d_x g_1, \d_{y_1}g_2, \d_{y_2}g_3,\ldots,\d_{y_{L-1}}g_L \neq 0$ and $\d_{y_i}g_j=0 \text{ for all } i \in\{1,\ldots,L\}, j \in\{i+2,\ldots,L+1\}$.
		\item[\emph{\textbf{(H2)}}] In $\tA$ we have $\d_xf\d^2_xg_1\neq\d^2_xf\d_xg_1$.
	\end{enumerate}
	If (H1) and (H2) are both valid, the local H{\"o}rmander condition holds at any $\Phi\in \tE^*$.
\end{thm}

\begin{remark}\label{rem:H1H2}
A more compact formulation of the cascade condition (H1) using \eqref{eq:FohneS} and \eqref{eq:xi} would be that in $\tA$ for all $i \in\{1,\ldots,L\}$ we have
\[
\d_{\xi_i}F_{i+1} \neq 0 \quad \text{and} \quad \d_{\xi_i}F_j=0 \text{ for all $j \in\{i+2,\ldots,L+1\}$},
\]
or in less strict but maybe more intuitive notation 
\[
\d_{\xi_i}F=\begin{pmatrix} \Box_{i} \\ \neq 0 \\ 0_{L-i}	\end{pmatrix} \quad  \text{for all $i \in\{1,\ldots,L\}$},
\]
where $\Box_i\in\R^i$ indicates that we do not care about the values of these first $i$ entries. Intuitively speaking, while the variable $\xi_i$ may or may not influence the first $i$ components of $F$, it definitely influences $F_{i+1}$ and definitely does not immediately influence any components afterwards -- hence the term cascade structure. This setting is mainly inspired by the diffusion given by (5.26) in \cite{Cascade}, which features two cascades that are similar in nature to the scenario we study in this section. There, the authors discuss an approximating diffusion for a model of interacting neurons that are divided into two groups. In each of these groups, a current is passed on from one neuron to the next, while only at the ends of these chains neurons are subject to noise and interact with neurons from the other group.

Condition (H2) serves the purpose of decoupling $X$ and $Z$, as will become clear in the proof of Theorem \ref{thm:cascade}. Note that if (H1) already holds,
\begin{enumerate}
	\item[\textbf{(H2')}] for all $(x,y)\in \tA$ the vectors $\d_x F(x,y)$ and $\d^2_x F(x,y)$ are linearly independent,
\end{enumerate}
is equivalent to (H2).
\end{remark}

In order to prove Theorem \ref{thm:cascade}, we take Lie brackets in a way that is somewhat dual to the approach we used in Section \ref{sect:star}. If we start with the vector field $V_1$ corresponding to the single column of the $(1+L+1)\times 1$-dimensional diffusion matrix and then keep taking Lie brackets with the vector field $V_0$ corresponding to the drift, intuition suggests that the structure from (H1) will turn up there again in some sense, providing more and more linearly independent vectors. So, this time we set
\begin{equation}\label{eq:cascadeL}
L_1:=[V_0,V_1], \quad \text{and} \quad L_n:=[V_0,L_{n-1}] \quad \text{for all $n\ge2$}.
\end{equation}
We have already calculated $L_1$ and $L_2$ in \eqref{eq:iacascadelemma1} and \eqref{eq:iacascadelemma2}. Using Lemma \ref{lem:liecascade}, we are able to get a better grasp of what $L_n$ looks like for larger $n$. This is the content of the following Lemma, which shows that the cascade structure in $F$ is indeed passed on to $L_1,\ldots,L_L$ in some sense.

\begin{lem}\label{lem:liecascade2}
	Let $n\in\{1,\ldots,L\}$ and assume that the cascade condition (H1) holds. Then for all $\xi=(t,x,y,z)\in [0,\infty)\times \tE^*$ we can write
	\begin{equation}\label{eq:cascadeLbar}
	L_n(\xi)=\begin{pmatrix} A_n(t,z)\\ 0_L \\ A_n(t,z) \end{pmatrix} - \begin{pmatrix} \bar L_n(\xi) \\ 0 \end{pmatrix},
	\end{equation}
	where
	\begin{equation}\label{eq:cascadeLbar2}
	\bar L_n(\xi):=\sum_{i=1}^{m_n-1} a_{n,i}(t,z) W_{n,i}(x,y) + (-1)^{n+1}\sigma(z)W_{n,m_n}(x,y)
	\end{equation}
	for suitable $m_n\in\N$ and smooth functions $a_{n,i}$, $A_n$, and $W_{n,i}$ with
	\[
	W_{n,i}=\begin{pmatrix} \Box_n \\ 0_{1+L-n}	\end{pmatrix} \quad  \text{for all $i \in\{1,\ldots,m_n-1\}$}, \quad W_{n,m_n}=\begin{pmatrix} \Box_n \\ \neq 0\\ 0_{1+L-(n+1)} \end{pmatrix}
	\]
	everywhere in $\tA$.
\end{lem}

\begin{proof}
	In analogy to Lemma \ref{lem:liestar2}, we prove this Lemma by induction. As seen in \eqref{eq:iacascadelemma1}, a representation of $L_n$ as in \eqref {eq:cascadeLbar} holds for $n=1$. We assume now that such a representation holds for some $n\in\{1,\ldots,L-1\}$. Set $a_{n,m_n}(t,z):=(-1)^{n+1}\sigma(z)$ for all $(t,z)\in[0,\infty)\times\tB$ and let $\xi=(t,x,y,z)\in [0,\infty)\times \tE^*$. Then, according to Lemma \ref{lem:liecascade}, we obtain
	\begin{align*}
		L_{n+1}(\xi)
		= & \begin{pmatrix} [\hat b(t,\cdot),A_n(t,\cdot)](z)+\d_t A_n(t,z)\\ 0_L \\ [\hat b(t,\cdot),A_n(t,\cdot)](z)+\d_t A_n(t,z) \end{pmatrix}  - \sum_{i=1}^{m_n} \d_ta_{n,i}(t,z) \begin{pmatrix} W_{n,i}(x,y) \\ 0 \end{pmatrix}\\
		&- A_n(t,z)\begin{pmatrix} \d_xF(x,y) \\ 0 \end{pmatrix} - \sum_{i=1}^{m_n}a_{n,i}(t,z) \begin{pmatrix} [F,W_{n,i}](x,y) \\ 0 \end{pmatrix} \\
		&- \sum_{i=1}^{m_n}\hat b(t,z)\left( a_{n,i}(t,z) \begin{pmatrix} \d_xW_{n,i}(x,y)\\ 0 \end{pmatrix} +\d_za_{n,i}(t,z) \begin{pmatrix}W_{n,i}(x,y)\\ 0 \end{pmatrix} \right).
	\end{align*}
	The first summand is already of the desired type, and every other summand has a vanishing last component. Therefore, we are left with the task to study the term
	\begin{align}\label{eq:lemliecascade2proof}
		\begin{split}
			\bar L_{n+1}(\xi)	:= & \sum_{i=1}^{m_n} \d_ta_{n,i}(t,z) W_{n,i}(x,y) + A_n(t,z)\d_xF(x,y) + \sum_{i=1}^{m_n}a_{n,i}(t,z) [F,W_{n,i}](x,y) \\
			&+ \sum_{i=1}^{m_n}\hat b(t,z)\Big( a_{n,i}(t,z) \d_xW_{n,i}(x,y)+\d_za_{n,i}(t,z) W_{n,i}(x,y)\Big).
		\end{split}
	\end{align}
	We have to prove that we can extract one summand from this sum that is of the type $(-1)^{n+2}\sigma(z)W(x,y)$ with some
	\[
	W(x,y)=\begin{pmatrix} \Box_{n+1} \\ \neq 0\\ 0_{1+L-(n+2)} \end{pmatrix},
	\]
	while every other summand is of the type $\phi(t,z)V(x,y)$ with some
	\[
	V(x,y)=\begin{pmatrix} \Box_{n+1} \\ 0_{1+L-(n+1)} \end{pmatrix}.
	\]
	In order to do so, we will expand the expressions in \eqref{eq:lemliecascade2proof} step for step and discuss the occurring terms one by one.
	
	Due to our induction hypothesis, we get that
	\[
	\text{$W_{n,i}(x,y)$, $\d_xW_{n,i}(x,y)$ are of the type $\begin{pmatrix} \Box_{n+1} \\ 0_{1+L-(n+1)}	\end{pmatrix}$ for all $i\in\{1,\ldots,m_n\}$,}
	\]
	and due to (H1), the same is true for $\d_xF(x,y)$. This is why it remains to look at
	\[
	\sum_{i=1}^{m_n} a_{n,i}(t,z)[F,W_{n,i}](x,y)= \sum_{i=1}^{m_n} a_{n,i}(t,z) \sum_{j=1}^{1+L} \left( F^{(j)} \d_{\xi_j} W_{n,i}- W_{n,i}^{(j)} \d_{\xi_j}F \right)(x,y).
	\]
	Thanks again to the induction hypothesis,
	\[
	\text{$\d_{\xi_j} W_{n,i}(x,y)=\begin{pmatrix} \Box_{n+1} \\ 0_{1+L-(n+1)}	\end{pmatrix}$ for all $i\in\{1,\ldots,m_n\}$ and $j\in\{1,\ldots,1+L\}$,}
	\]
	which is of course still valid for the same terms multiplied by $F^{(j)}(x,y)$. Therefore, it remains to discuss
	\[
	-\sum_{i=1}^{m_n} a_{n,i}(t,z) \sum_{j=1}^{1+L} W_{n,i}^{(j)}(x,y) \d_{\xi_j}F(x,y).
	\]
	Using the induction hypothesis once again, we know in particular that
	\[
	W_{n,i}^{(j)}(x,y)=0 \quad \text{for all $i\in\{1,\ldots,m_n-1\}$ and $j\in\{n+1,\ldots,1+L\}$}
	\]
	and
	\[
	W_{n,m_n}^{(j)}(x,y)=0 \quad \text{for all $j\in\{n+2,\ldots,1+L\}$.}
	\]
	Thus, it remains to look at
	\[
	-\sum_{i=1}^{m_n-1} a_{n,i}(t,z) \sum_{j=1}^{n} W_{n,i}^{(j)}(x,y) \d_{\xi_j}F(x,y) - (-1)^{n+1}\sigma(z) \sum_{j=1}^{n+1} W_{n,m_n}^{(j)}(x,y) \d_{\xi_j}F(x,y).
	\]
	Using (H1) again, we see that
	\[
	\d_{\xi_j}F(x,y) = \begin{pmatrix} \Box_{n+1} \\ 0_{1+L-(n+1)}	\end{pmatrix} \quad \text{for all }j\in\{1,\ldots,n\},
	\]
	which implies that it now remains to comment on the summand
	\[
	(-1)^{n+2}\sigma(z) W_{n,m_n}^{(n+1)}(x,y) \d_{\xi_{n+1}}F(x,y).
	\]
	The two leading scalar factors do not vanish (by \eqref{eq:sigmapos} and by the induction hypothesis), and
	\[
	\d_{\xi_{n+1}}F(x,y)=\begin{pmatrix} \Box_{n+1} \\ \neq 0\\ 0_{1+L-(n+2)} \end{pmatrix}
	\]
	thanks to (H1). Thus, $\bar L_{n+1}(\xi)$ is in fact of the desired form with
	\[
	W_{n+1,m_{n+1}}(x,y):=W_{n,m_n}^{(n+1)}(x,y) \d_{\xi_{n+1}}F(x,y)=\begin{pmatrix} \Box_{n+1} \\ \neq 0\\ 0_{1+L-(n+2)} \end{pmatrix},
	\]
	and the proof by induction is complete.
\end{proof}

We can now turn to the

\begin{proof}[Proof of Theorem \ref{thm:cascade}]
Assumption (H1) and Lemma \ref{lem:liecascade2} immediately imply that for arguments from $[0,\infty)\times \tE^*$ the vector fields $\bar L_1,\ldots,\bar L_L$ are of the type
\[
	\bar L_n=\begin{pmatrix} \Box_n \\ \neq 0\\ 0_{1+L-(n+1)} \end{pmatrix} \quad  \text{for all $n\in\{1,\ldots,L\}$,}
\]
which in turn yields their linear independence. Lemma \ref{lem:liecascade2} also implies that $L_n$ differs from
\[
-\begin{pmatrix} \bar L_n(\xi) \\ 0 \end{pmatrix}
\]
only by a multiple of the diffusion coefficient $V_1(z)$, which is nonzero since $z\in\tB$. Combining these two facts lets us conclude that
\[
V_1(z),L_1(\xi),\ldots,L_L(\xi)
\]
are linearly independent. In order to span the remaining dimension, we set
\begin{align*}
	L_{L+1}(\xi):=& [V_1,[V_1,V_0]](\xi)=-[V_1,L_1](\xi) \\
	=& \begin{pmatrix} \zeta_{1,1}(t,z) \\ 0_L \\ \zeta_{1,1}(t,z) \end{pmatrix}+ \sigma^2(z)\begin{pmatrix}\d^2_x f(x,y) \\ \d^2_x g_1(x,y) \\ 0_L	\end{pmatrix} + \sigma(z)\sigma'(z) \begin{pmatrix}\d_x f(x,y) \\ \d_x g_1(x,y) \\ 0_L	\end{pmatrix},
\end{align*}
where we used \eqref{eq:iastarlemma1} and the cascade condition. Using similar arguments as in the proof of Corollary \ref{cor:stardiag}, we see after subtracting a suitable linear combination of $V_0(z)$ and $L_1(\xi)$ that $L_{L+1}(\xi)$ is linearly independent from $V_1(z),L_1(\xi),\ldots,L_L(\xi)$ whenever the decoupling condition (H2) is fulfilled.
\end{proof}

\begin{remark}
	Note that Lemma \ref{lem:liestar2} provides a general representation of $L_{\kappa,n}$ while specific conditions on $f$ and $g$ come in only afterwards. In contrast to this, the calculations of $L_n$ in Lemma \ref{lem:liecascade2} already rely explicitly on the condition (H1), without which the representation would become very bulky and rather unhelpful. This slightly different approach is also reflected in the fact that in the present section the assumptions on $f$ and $g$ are made not only in a single point but in the open neighborhood $\tA$.	
\end{remark}

\begin{example}
For $M=N=1$ and the choice of $g$ that is given in \eqref{eq:toyg2}, the system from Example \ref{ex:toy} obviously fulfills (H1) and (H2) in a neighborhood $\tA$ of $(\pm1,1,\ldots,1)$, and $\tB$ is non-empty unless $\sigma\equiv0$. In Example \ref{ex:TOYcontrolorbit}, we showed that $(\pm1,1,\ldots,1,z)$ is also $T$-attainable for any $z\in\R$. Since, as noted in Remark \ref{rem:lyapunov}, there is also a Lyapunov function in this setting, this means that all of the assumptions of Theorem \ref{thm:HLT} are satisfied and the solution to the corresponding SDE \eqref{eq:SDE} is positive Harris recurrent (even exponentially ergodic).
\end{example}

\begin{example}
	Even though it might seem like the oscillator model from Example \ref{ex:rotors} should involve some kind of cascade structure (at least in the setting with one-sided input), calculating the derivatives of the respective function $g$ shows that it actually does not (even after renumbering the variables). If we apply the strategy from \eqref{eq:cascadeL} anyway and do calculations in analogy to Lemma \ref{lem:liecascade2}, we can span $N+L=1+5$ space directions in the one-sided case at points where $w_1'(q_2-q_1)$ and $w_3'(q_2-q_3)$ do not vanish. For the case of two-sided input, we have $N=2$ and therefore we have to do the same construction twice: Once just as in \eqref{eq:cascadeL} and once with $V_1$ replaced with $V_2$. For the respective calculations, recall that Lemma \ref{lem:liecascade} works for any $N\in\N$. Again, we can span $N+L=2+4$ space directions at points where $w_1'(q_2-q_1)$ and $w_3'(q_2-q_3)$ do not vanish. For the time-homogeneous system from \cite{Cuneo} without explicit external equations, one can allow those derivatives to vanish as long as some higher order of them does not, as the authors prove in \cite[Lemma 5.3]{Cuneo}. In our case however, time-inhomogeneity and the extra dependence on $z$ in the drift make this more difficult. Decoupling of $X$ and $Z$ is even more problematic, since any dependence on the $x$-variable (i.e. $p_1$ or $(p_1,p_3)$, respectively) is linear, which rules out working with a condition like (H2). At the moment, we do not know how to provide a method that will work under assumptions that are neither too restrictive nor physically irrelevant.
\end{example}

\noindent\textbf{Acknowledgements.} The author would like to thank Reinhard H{\"o}pfner, Eva L{\"o}cherbach and Barbara Gentz for fruitful discussions and helpful remarks and suggestions.

\end{document}